\RequirePackage{fix-cm}
\documentclass[10pt]{amsart}       % onecolumn (second format)

\usepackage{amsmath, amsthm, amssymb, amsfonts, verbatim}
\usepackage{array}
\usepackage[bookmarks]{hyperref}
\usepackage[active]{srcltx}
\usepackage{amscd}
\usepackage{epsfig}
\usepackage{color}
\usepackage{comment}
\usepackage{setspace}
\usepackage{textcomp}
\usepackage{stmaryrd}
\usepackage{multirow}
\newcommand{\pmat}[1]{\begin{pmatrix} #1 \end{pmatrix}}

\newcommand{\beq} {\begin{equation}}
\newcommand{\eeq} {\end{equation}}
\newcommand{\bdm} {\begin{displaymath}}
\newcommand{\edm} {\end{displaymath}}
\newcommand{\bitem}[1]{\begin{itemize} #1 \end{itemize}}
\newcommand{\bit}{\begin{itemize}}
\newcommand{\eit}{\end{itemize}}
\newcommand{\bde}{\begin{description}}
\newcommand{\ede}{\end{description}}

\newcommand{\ben}{\begin{enumerate}}
\newcommand{\een}{\end{enumerate}}
\newcommand{\algn}[1]{\begin{align} #1 \end{align}}
\newcommand{\algns}[1]{\begin{align*} #1 \end{align*}}
\newcommand{\mltln}[1]{\begin{multline} #1 \end{multline}}
\newcommand{\mltlns}[1]{\begin{multline*} #1 \end{multline*}}
\newcommand{\gat}[1]{\begin{gather} #1 \end{gather}}
\newcommand{\gats}[1]{\begin{gather*} #1 \end{gather*}}
\newcommand{\barr}{\begin{array}}
\newcommand{\earr}{\end{array}}

\newcommand{\half} {\ensuremath{\frac{1}{2}}}
\newcommand{\mc}[1]{\mathcal{#1}}
\newcommand{\LRp}[1]{\left( #1 \right)}

\newcommand{\LRa}[1]{\left< #1 \right>}

\newcommand{\ra}{\rightarrow}

\newcommand{\e}{\epsilon}

\newcommand{\tred}[1]{\textcolor{black}{#1}}
\newcommand{\tblu}[1]{\textcolor{black}{#1}}

\newcommand{\R}{{\mathbb R}}

\newcommand{\V}{{\mathbb V}}
\newcommand{\K}{{\mathbb K}}

\newcommand{\M}{{\mathbb M}}
\newcommand{\X}{{\mathbb X}}
\newcommand{\I}{\Bbb{I}}

\newcommand{\grad}{\operatorname{grad}}

\renewcommand{\div}{\operatorname{div}}

\newcommand{\tr}{\operatorname{tr}}

\newcommand{\esssup}{\operatorname{esssup}}

\newcommand{\bs}{\boldsymbol}

\newcommand{\lap}{\Delta}
\newcommand{\pd}{\partial}

\newtheorem{theorem}{Theorem}[section]
 
\newtheorem{lemma}[theorem]{Lemma}
\newtheorem{cor}[theorem]{Corollary}

\newtheorem{rmk}[theorem]{Remark}
\newtheorem{eg}[theorem]{Example}

\begin{document}

\title[]{Robust error analysis of coupled mixed methods for Biot's consolidation model}\thanks{The work of Jeonghun~J.~Lee has been supported by the European Research Council under the European Union's Seventh Framework Programme (FP7/2007-2013) ERC grant agreement 339643 (PI : Prof. Ragnar Winther).} 
%The work of Kent-Andre~Mardal has been supported by the Research Council of Norway through grant no. 209951 and a Center of Excellence grant awarded to the Center for Biomedical Computing at Simula Research Laboratory.}
\author{Jeonghun J. Lee }%\footnote{Department of Mathematics, University of Oslo, P.O. Box 1053, Blindern 0316, Norway} }
%P.O. Box 11100, 00076 Aalto, Finland \\
%mika.juntunen@aalto.fi}} }
\address{
Department of Mathematics, University of Oslo \\
P.O. Box 1053, Blindern 0316, Oslo, Norway \\
jeonghul@math.uio.no/johnlee04@gmail.com}

\maketitle
% \institute{Jeonghun J. Lee \at
%               Department of Mathematics, University of Oslo, P.O.box 1053, Blindern, 0316, Oslo, Norway \\
%               Tel.: +47-22-855691\\
%               Fax: +47-22-855886\\
%               \email{jeonghul@math.uio.no}           %  \\
% %             \emph{Present address:} of F. Author  %  if needed
% } The correct dates will be entered by the editor

\maketitle

\begin{abstract}
We study the a priori error analysis of finite element methods for Biot's consolidation model. We consider a formulation which has the stress tensor, the fluid flux, the solid displacement, and the pore pressure as unknowns. Two mixed finite elements, one for linear elasticity and the other for mixed Poisson problems are coupled for spatial discretization, and we show that any pair of stable mixed finite elements is available. The novelty of our analysis is that the error estimates of all the unknowns are robust for material parameters. Specifically, the analysis does not need a uniformly positive storage coefficient, and the error estimates are robust for nearly incompressible materials. Numerical experiments illustrating our theoretical analysis are included.
%Insert your abstract here. Include keywords, PACS and mathematical subject classification numbers as needed.
%\keywords{Poroelasticity\and error analysis\and mixed finite elements}
% \PACS{PACS code1 \and PACS code2 \and more}
%\subclass{65N30 \and 65N12}
\end{abstract}

\section{Introduction}
Biot's consolidation model describes the deformation of an elastic porous medium and the viscous fluid flow inside it when the medium is saturated by the fluid \cite{MR0066874}. This model has many applications in various engineering fields including geomechanics, petrolieum engineering, and biomedical engineering. 

There are numerous studies of numerical schemes for Biot's consolidation model with finite element methods%In early studies of the problem with continuous Galerkin finite elements, nonphysical pressure oscillations of numerical solutions, called poroelasticity locking, were observed for certain ranges of material parameters and small time step sizes 
. 
%In order to avoid the poroelasticity locking, various numerical methods for the problem with different formulations were considered. 
In a series of papers \cite{MR1156589,MR1257948,MR1393902}, Murad et al. studied a formulation with the solid displacement and the pore pressure as unknowns using mixed finite elements for the Stokes equation. A discontinuous Galerkin method for the same formulation was also studied in \cite{MR3047799}. A Galerkin least square method was proposed for a formulation with four unknowns, i.e., the solid displacement, a pseudo-stress tensor, the fluid flux, and the pore pressure \cite{MR2177147}. A formulation with the solid displacement, the fluid flux, and the pore pressure as unknowns was studied with various couplings of continuous and discontinuous Galerkin methods, and mixed finite element methods \cite{MR2327964,MR2327966,MR2461315}.  A coupling of nonconforming and mixed finite element methods for the formulation was recently studied in \cite{NUM:NUM21775}. For more information on previous studies we refer to \cite{MR3047799,lewis1998finite,NUM:NUM21775} and the references therein.
%, Murad et.al. proved decay of numerical pressure oscillation in time. 
%is proved when external load has a limit as time tends to infinity.There are a least square method \cite{MR2177147}, a stabilized CG method (Wan), coupling mixed and CG methods \cite{MR2327964,MR2327966}, discontinuous Galerkin (DG) methods (Liu), coupling mixed and DG methods \cite{MR2461315}. Among others, a coupling of nonconforming and mixed finite elements \cite{NUM:NUM21775}, and a stabilized DG method \cite{MR3047799} are proposed recently. 

In this paper we consider a formulation with four unknowns, i.e., the stress tensor, the solid displacement, the fluid flux, the pore pressure. This was considered recently in \cite{MR3200272} using a combination of two mixed finite elements for the discretization of the problem, one for linear elasticity with symmetric stress tensors and the other for mixed Poisson problems. A numerical experiment in the paper shows that this approach can be advantageous to avoid non-physical pressure oscillations when the constrained storage coefficient of the problem vanishes. In this paper we also use a combination of two mixed finite elements for the discretization but we use mixed finite elements for linear elasticity with weakly symmetric stress because the elements with weakly symmetric stress can be advantageous with respect to efficient implementation and low computational costs. The main contribution of this paper is {\it a new error analysis providing the a priori error estimates that are robust for material parameters}. More specifically, we give error estimates of all the unknowns with the $L^\infty$ norm in time and $L^2$ norm in space, and the estimates do not need strict positivity of the constrained storage coefficient $s_0$. Moreover, as in the Hellinger--Reissner formulation of linear elasticity \cite{ADG84}, the error bounds are uniform for the parameter indicating incompressibility of the poroelastic medium, i.e., the error estimates are robust for nearly incompressible materials. To the best of our knowledge, an analytic proof of this robustness for incompressibility has not been addressed in literature.

The paper is organized as follows. In section \ref{sec:prelim} we define notation \tblu{and} derive a variational formulation of the problem. In section \ref{sec:semi} we present finite element methods for the semidiscrete problem and show the a priori error analysis of semidiscrete solutions. We also prove robustness of the error estimate for nearly incompressible materials and show well-posedness of fully discrete solutions with the backward Euler discretization in time. In section \ref{sec:numerical} numerical results illustrating our theoretical analysis are presented.

\section{Biot model and variational formulations} \label{sec:prelim}

\subsection{Notation}
Let $\Omega$ be a bounded Lipschitz domain in $\R^n$ with $n=2$ or $3$. Let $L^2(\Omega)$ be the set of square-integrable real-valued functions on $\Omega$. The inner product of $L^2(\Omega)$ and the induced norm are denoted by $(\cdot, \cdot)$ and $\| \cdot \|_0$. For a finite-dimensional inner product space $\X$, let $L^2(\Omega; \X)$ be the set of $\X$-valued functions such that each component of the functions is in $L^2(\Omega)$. The inner product of $L^2(\Omega; \X)$ is naturally defined by the inner product of $\X$ and $L^2(\Omega)$, so we use the same notation $(\cdot, \cdot)$ and $\| \cdot \|_0$ to denote the inner product and norm on $L^2(\Omega; \X)$. The inner product space $\X$ is the space of $\R^n$ vectors with standard inner product or a subspace of $n \times n$ matrices with the Frobenius inner product. For future reference we use $\M$, $\Bbb{S}$, $\K$ to denote the spaces of all, symmetric, skew-symmetric $n \times n$ matrices, respectively, and $\V$ to denote the space of column $\R^n$ vectors. 

For a nonnegative integer $m$, $H^m (\Omega)$ denotes the standard Sobolev spaces of real-valued functions based on the $L^2$ norm, and $H^m(\Omega; \Bbb{X})$ is defined similarly based on $L^2(\Omega; \X)$. For $m \geq 1$, we use ${H}_0^m(\Omega)$ to denote the subspace of $H^m(\Omega)$ with vanishing trace on $\pd \Omega$ \cite{Evans-Book}, and ${H}_0^m(\Omega; \X)$ is defined similarly. For simplicity we also use $\| \cdot \|_m$ to denote the $H^m$-norm for both of $H^m(\Omega)$ and $H^m(\Omega; \X)$. We define $H(\div, \Omega)$ as 
\algns{
H(\div, \Omega) := \{ v \in L^2(\Omega; \R^n) \;:\; \div v \in L^2(\Omega) \},
}
%The set of functions in $L^2(\Omega; \R^n)$ whose divergence are in $L^2(\Omega)$, 
and the norm $(\| v \|_0^2 + \| \div v \|_0^2)^{1/2}$ is denoted by $\| v \|_{\div}$. We also define the space $H(\div, \Omega; \M)$ and $H(\div, \Omega; \Bbb{S})$ as 
\algns{
H(\div, \Omega; \M) &:= \{ \tau \in L^2(\Omega; \M) \;:\; \div \tau \in L^2(\Omega; \V) \}, \\
H(\div, \Omega; \Bbb{S}) &:= H(\div, \Omega; \M) \cap L^2(\Omega; \Bbb{S}),
}
in which the divergence of $\tau$ is understood as the row-wise divergence of $\tau$, and $\| \tau \|_{\div}$ is defined similarly as the norm of $H(\div)$. 

%For $\X=\V, \M, \K$ or $\mathbb{S}$, $H^m(\Omega; \X)$ is the space of $\X$-valued functions such that each component of a function is in $H^m(\Omega)$. If $\X$ is clear in context, $H^m (\Omega)$ is used as an abbreviation of $H^m (\Omega; \X)$. % for $Y = \M$, $\mathbb S$, $\K$, $\V$. 

Let $J = [0, T]$, $T >0$ be an interval. For a reflexive Banach space $\mathcal{X}$, let $C^0 (J ; \mathcal{X})$ denote the set of functions $f : J \ra \mathcal{X}$ which are continuous in $t \in J$. For an integer $m \geq 1$, we define 
\begin{align*}
C^m (J; \mathcal{X}) = \{ f \, | \, \pd^{i}f/\pd t^{i} \in C^0(J;\mathcal{X}), \, 0 \leq i \leq m \},
\end{align*}
where $\pd^i f/\pd t^i$ is the $i$-th time derivative in the sense of the Fr\'echet derivative in $\mathcal{X}$ (see e.g., \cite{Yosida-book}). For a function $f : J \ra \mathcal{X}$, we define the space-time norm %function space $L^p([0,T]; X)$ is defined by the Banach space endowed with the norm 
\begin{align*}
\| f \|_{L^r(J; \mathcal{X})} = 
\begin{cases}
\left( \int_J \| f \|_\mathcal{X}^r ds \right)^{1/r}, \quad 1 \leq r < \infty, \\
\esssup_{t \in J} \| f \|_\mathcal{X}, \quad r = \infty.
\end{cases}
\end{align*}
%Note that the above definition can be employed to the semi-normed spaces and $|f|_{L^p([a,b]; \dot H^m)}$ is defined similarly as above by using $|f|_m$ instead of $\| f \|_X$. 
If the time interval is fixed, then we use $L^r \mathcal{X}$ instead of $L^r(J ; \mathcal{X})$ for simplicity. We define the space-time Sobolev spaces $W^{k,r}(J; \mathcal{X})$ for a nonnegative integer $k$ and $1 \leq r \leq \infty$ as the closure of $C^k (J; \mathcal{X})$ with the norm $\| f \|_{W^{k,r} \mathcal{X}} = \sum_{i=0}^k \| \pd^i f / \pd t^i \|_{L^r \mathcal{X}}$. 
% The Sobolev embedding \cite{MR697382} gives
% \begin{align} \label{eq:sobolev}
% W^{k+1,1} \mathcal{X} \hookrightarrow W^{k,\infty} \mathcal{X}. 
% \end{align}
We adopt a convention that $\| f, g \|_\mathcal{X} = \| f \|_\mathcal{X} + \| g \|_\mathcal{X}$ for the norm of a Banach space $\mathcal{X}$. 
%For two Banach spaces $\mathcal{X}$ and $\mathcal{Y}$, and for $f \in \mathcal{X} \cap \mathcal{Y}$, $\| f \|_{\mathcal{X} \cap \mathcal{Y}}$ will stand for $\| f \|_{\mathcal{X}} + \| f \|_{\mathcal{Y}}$. 
For simplicity of notation, $\dot{f}$ will be used to denote $\pd f/ \pd t$. 

%Similarly, $\ddot{\sigma}$, $\dddot{\sigma}$ are used to denote ${\pd^2 \sigma}/ { \pd t^2}$, ${\pd^3 \sigma}/ {\pd t^3}$, respectively. 

% A shape-regular triangulation of $\Omega$ will be denoted by $\calT_h$ for which $h$ is the maximum diameter of triangles (or tetrahedra) and $\mathcal{E}_h$ is the corresponding set of edges (faces), respectively. The interior edges/faces $\mathcal{E}_h^{\circ}$ is the set $\{ E \in \mathcal{E}_h \,|\, E \subset \Omega \}$. For $E \in \mathcal{E}_h$ and functions $\bs{f}, \bs{g} : \mathcal{E}_h \ra \R^n$ we define
% \begin{align*}
% \langle \bs{f}, \bs{g} \rangle_E = \int_E \bs{f} \cdot \bs{g} \,ds, \qquad \langle \bs{f}, \bs{g} \rangle = \sum_{E \in \mathcal{E}_h} \langle \bs{f}, \bs{g} \rangle_E. 
% \end{align*}
% For $E \in \mathcal{E}_h$ and an element-wise $H^1$ function $\bs{v}$, $\lb \bs{v} \rb$ is defined by 
% \begin{align*}
% \lb \bs{v} \rb|_E = 
% \begin{cases}
% \text{the jump of }\bs{v} \text{ on } E, \quad &\text{ if } E \in \mathcal{E}_h^{\circ}, \\
% \bs{v}, \quad &\text{ if }E \subset \pd \Omega.
% \end{cases}
% \end{align*}
% For an integer $k \geq 0$, and $G \subset \R^n$, $\mathcal{P}_k(G)$ is the space of polynomials defined on $G$ of degree $\leq k$. We use $\calP_k (\calT_h)$ to denote the space of piecewise polynomials on $\mathcal{T}_h$ of degree $\leq k$. For a vector space $\X$, we use $\mathcal{P}_k(G; \X)$ and $\mathcal{P}_k(\mathcal{T}_h; \X)$ to denote the space of $\X$-valued polynomials with same conditions. 

Finally, throughout this paper we use $X \lesssim Y$ to denote the inequality $X \leq cY$ with a generic constant $c>0$ which is independent of mesh sizes. 
If needed, we will write $c$ explicitly in inequalities but it can be different in each formula.
%If $\X$ is a subspace of $\M$, then $\calP_k (\calT_h, \div; \X)$ is the subspace of $\calP_k (\calT_h; \X)$ which is in $H(\div, \Omega; \X)$. %For classical finite elements, we use $RT_k$, $BDM_k$, $DG_k$ to denote the Raviart--Thomas (\cite{RT75}) space of degree $k$, the Brezzi--Douglas--Marini space (\cite{BDM85}) of degree $r$ and the space of (not necessarily continuous) piecewise polynomials of degree $r$. % , \cite{Nedelec86}). 

\subsection{Biot's consolidation model} % and variational formulation} \label{sec:model}
In this subsection we derive a formulation of Biot's consolidation model with four unknowns and establish a variational formulation of the problem. %In an elastic porous medium saturated with a fluid, fluid flow and deformation of the porous medium are intimately related and their simultaneous behaviors are described by Biot's model.%when permeability of the medium is small, i.e., deformation of the medium and retardation of fluid flow occur simultaneously. In saturated porous media it is described by Biot's consolidation model. 
%Throughout this article we restrict our interest to quasistatic consolidation problems. %In other words, we assume that the consolidation process is slow and the acceleration term is ignored. 

Let $\Omega$, a bounded Lipschitz domain in $\R^n$ with $n=2$ or $3$, be occupied by a fluid-saturated poroelastic body. Let $u :\Omega \ra \V$ be the displacement of the poroelastic medium, $p : \Omega \ra \R$ the pore pressure, $f : \Omega \ra \V$ the body force, and $g : \Omega \ra \R$ the source/sink density function of the fluid. The governing equations of Biot's consolidation model are 
\begin{align}
\label{eq:2-field-eq1}  - \div \mathcal{C} \e(u) + \alpha \nabla p &= f & & \text{in } \Omega,\\
\label{eq:2-field-eq2} s_0 \dot{p} + \alpha \div \dot{u} - \div (\kappa \nabla p) &= g & & \text{in } \Omega, 
\end{align}
where $\mathcal{C}$ is the elastic stiffness tensor, $\e(u)$ is the linearized strain tensor, $s_0 \geq 0 $ is the constrained specific storage coefficient, $\kappa$ is the hydraulic conductivity tensor, and $\alpha>0$ is the Biot--Willis constant which is close to 1. 
% In dynamic model, \eqref{eq:2-field-eq1} is replaced by 
% \begin{align}
% \label{eq:2-field-eq1d} \gamma_0 \ddot{u} - \div \mathcal{C} \e(u) + \alpha \nabla p &= f.
% \end{align}
% with $\gamma_0$, the density of elastic medium, 
In order to understand the system \eqref{eq:2-field-eq1}--\eqref{eq:2-field-eq2} precisely, we need to explain the operators. First, by $\grad u$ we mean the $\M$-valued function such that each row is the gradient of each component of $u : \Omega \ra \V$, and $\e(u)$ is the symmetric matrix part of $\grad u$. There is no confusion of the divergence operator $\div$ for vector-valued functions but, when it is used for $\M$-valued functions, $\div$ will be ($n$-tuples of) the row-wise divergence of the $\M$-valued function which results in a $\V$-valued function. If $q$ is a scalar function, then $\nabla q$ stands for the gradient of $q$ as a column vector. With these conventions the equations in the system \eqref{eq:2-field-eq1}--\eqref{eq:2-field-eq2} are well-defined.

In general the elastic stiffness tensor $\mathcal{C}$ is a rank $4$ tensor giving a symmetric positive definite linear map from $L^2(\Omega; \Bbb{S})$ into itself \cite{Gurtin-book-81}. The coefficient $s_0 \geq 0 $ is determined by material parameters such as the permeability (of the porous medium) and the bulk moduli of the solid and the fluid. The hydraulic conductivity tensor $\kappa$ is defined by the permeability tensor of the solid divided by the fluid viscosity and it is positive definite. All the parameters $\mc{C}$, $s_0$, $\kappa$, and $\alpha$ are functions of $x \in \Omega$. For the derivation of these equations from physical modeling, we refer to standard porous media references, for example, \cite{anandarajah2010computational,deBoer}.

To derive a formulation with four unknowns we introduce the fluid flux $z = \kappa \nabla p$ and the stress tensor $\sigma = \mathcal{C} \e(u) - \alpha p \I$ as new unknowns, where $\I$ is the $n\times n$ identity matrix. By the definitions of $\sigma$ and $z$, we have 
\begin{align}
\label{eq:4-field-eq1} \mathcal{A}^s (\sigma + \alpha p \I) - \e(u) &= 0, \\
\label{eq:4-field-eq2} \kappa^{-1} z - \nabla p  &= 0, 
\end{align}
where $\mc{A}^s = \mc{C}^{-1}$, and we can rewrite \eqref{eq:2-field-eq1} as
\begin{align}
\label{eq:4-field-eq3} - \div \sigma &= f .
\end{align}
In addition, observing that $\div {u} = \tr \e({u}) = \tr \mc{A}^s (\sigma + \alpha p \I)$ where $\tr$ is the trace of matrices, we can rewrite \eqref{eq:2-field-eq2} as 
\begin{align}
\label{eq:4-field-eq4} s_0 \dot{p} + \alpha \tr \mc{A}^s ( \dot \sigma + \alpha \dot{p} \I) + \div z &= g.
\end{align}
As a consequence, we obtain a system with four unknowns $\sigma$, $u$, $z$, $p$, and four equations \eqref{eq:4-field-eq1}--\eqref{eq:4-field-eq4}. In order to be a well-posed problem, the equations \eqref{eq:4-field-eq1}--\eqref{eq:4-field-eq4} need appropriate boundary and initial conditions. We assume that there are two partitions of $\pd \Omega$,
\begin{align*}
\pd \Omega = \Gamma_p \cup \Gamma_f, \qquad \pd \Omega = \Gamma_d \cup \Gamma_t,
\end{align*}
with $| \Gamma_p |,| \Gamma_d | > 0$, i.e., the Lebesgue measures of $\Gamma_p$ and $\Gamma_d$ are positive. Let $\bs{n}$ be the outward unit normal vector field on $\pd \Omega$. 
% \algn{ \label{eq:sigma-eq}
% \sigma(t) := \mathcal{C} \e(u(t)) - \alpha p(t) \I ,
% }
%where $\I$ is the $n\times n$-matrix. 
Boundary conditions are given in general by
\tblu{
\begin{align}
\label{eq:bc1}  
%\begin{split}
p(t) &= p_0 (t) \quad \text{ on } \Gamma_p, &z(t) \cdot \bs{n} &= z_{\bs{n}} (t) \quad \text{ on } \Gamma_f, \\
\label{eq:bc2}  
u(t) &= u_0 (t) \quad \text{ on } \Gamma_d, &\sigma(t) \bs{n} &= \sigma_{\bs{n}} (t) \quad \text{ on } \Gamma_t,
%\end{split}
\end{align}}
for all $t \in J$ with given 
\algns{
p_0 : J \times \Gamma_p \ra \R, \quad z_{\bs{n}} : J \times \Gamma_f \ra \R, \quad u_0 : J \times \Gamma_d \ra \V, \quad \sigma_{\bs{n}} : J \times \Gamma_t \ra \V .
}
% \begin{align}
% \label{eq:bc1}  
% %\begin{split}
% p(t) = 0 \quad &\text{ on } \Gamma_p, \qquad \qquad - \kappa \nabla p(t) \cdot \bs{n} = 0 &\text{ on } \Gamma_f, \\
% \label{eq:bc2}  
%  u(t) = 0 \quad &\text{ on } \Gamma_d, \qquad \qquad \sigma(t) \bs{n} = 0 &\text{ on } \Gamma_t,
% %\end{split}
% \end{align}
%Here we only consider this homogeneous boundary condition for simplicity but our method, which will be introduced later, can be extended readily to problems with inhomogeneous boundary conditions. We also assume that given initial data $p(0), \bs{u}(0)$ and initial body force $\bs{f}(0)$ satisfy \eqref{eq:strong-eq1}.
For well-posedness of \eqref{eq:4-field-eq1}--\eqref{eq:4-field-eq4} with the above boundary conditions, we can adopt the argument in \cite{MR1790411} used for well-posedness of the system \eqref{eq:2-field-eq1}--\eqref{eq:2-field-eq2} with same boundary conditions. Namely, we find expressions of $\sigma$ and $z$ in terms of $p$ using \eqref{eq:4-field-eq1}--\eqref{eq:4-field-eq3}, and apply these expressions to \eqref{eq:4-field-eq4} to obtain a parabolic partial differential equation of $p$. Since well-posedness of the system is not our main interest, we will not pursue it further in this paper.
%For simplicity we assume that $\Gamma_d = \pd \Omega = \Gamma_p$, and $p_0$, $u_0$ in \eqref{eq:bc1}--\eqref{eq:bc2} vanish. Using function spaces 
% \algns{
% H(\div, \Omega; \Bbb{S}), \quad \mathring{H}^1(\Omega;\V), \quad H(\div, \Omega), \quad \mathring{H}^1 (\Omega)
% }
% for $\sigma$, $u$, $z$, $p$, 
%\eqref{eq:2-field-eq1},\eqref{eq:2-field-eq2} and \eqref{eq:2-field-eq1}, \eqref{eq:2-field-eq1d} with suitable boundary conditions were studied in \cite{MR1790411}. %According to the results in \cite{MR1790411} regularity of solutions can be low. 
%When we claim a priori error estimates we assume that exact solutions are regular enough to let the claimed error bounds make sense. For simplicity, we assume that $\alpha=1$ in the rest of the paper. However, we assume that $c_0 \geq 0$ is only bounded from above.

% because we are mainly concerned with robust error bounds for general $c_0 \geq 0$. In previous studies it is observed that the poroelasticity locking occurs when the $c_0$, the hydraulic conductivity, and time step are very small.

%\subsection{A formulation with coupled mixed methods}
%In addition to \eqref{eq:2-field-eq1}--\eqref{eq:2-field-eq2}, there are other formulations of the system which can be preferred depending on the unknowns of primary interest. 

For spatial discretization we use two mixed finite element methods, one for linear elasticity of the Hellinger--Reissner formulation and the other for mixed Poisson problems. For discretization of linear elasticity we will use mixed finite elements for elasticity with weakly symmetric stress. Compared to mixed finite elements with symmetric stress tensors, the elements with weakly symmetric stress can be preferable because they usually require less computational costs and can be implemented with the Fraejis~de~Veubeke hybridization, which results in a system with reduced sizes \cite{AB85,Fraejis}. %From implementation point of view, they can be implemented easily with FEniCS.

In order to use mixed finite elements for elasticity with weakly symmetric stress, we introduce the skew-symmetric part of $\grad u$, denoted by $\gamma$, as another unknown. This new unknown plays a role of a Lagrange multiplier for the symmetry of the stress tensor. To formulate it, we define $\mc{A}$ as an extension of $\mc{A}^s$ such that $\mc{A} = \mc{A}^s$ on $L^2(\Omega; \mathbb{S})$ and $\mc{A}$ is a positive constant multiple of the identity map on $L^2(\Omega; \mathbb{K})$, where $\Bbb{K}$ is the space of $n \times n$ skew-symmetric matrices. Then, recalling that $\gamma$ is the skew-symmetric part of $\grad u$, \eqref{eq:4-field-eq1} can be written as 
\algn{
\label{eq:four-field-eq2} \mc{A} (\sigma + \alpha p \I) - \grad u + \gamma = 0 ,
}
and we have the symmetry constraint of $\sigma$, 
\algn{
(\sigma, \eta) = 0 \qquad \forall \eta \in L^2(\Omega; \K).
}
Let us define the function spaces 
\gat{
\label{eq:svgamma-space} \Sigma = H(\div, \Omega; \M), \qquad V = L^2(\Omega;\V), \qquad \Gamma = L^2(\Omega; \Bbb{K}), \\
\notag W = H(\div, \Omega), \qquad Q = L^2(\Omega) ,
}
for unknowns $(\sigma, u, \gamma, z, p)$. Then, by integration by parts with vanishing boundary conditions, we can derive the following variational formulation: Find 
\begin{align*}
(\sigma, p) \in C^1(J; \Sigma \times Q) \quad \text{and} \quad (u, \gamma, z) \in C^0(J; V \times \Gamma \times W), 
\end{align*}
such that 
\tblu{
\begin{align}
\label{eq:weak-eq1} (\mathcal{A} (\sigma + \alpha p \I), \tau ) + (u, \div \tau) + (\gamma, \tau) &= 0, & & \forall \tau \in \Sigma, \\
\label{eq:weak-eq2} (\div \sigma, v) + (\sigma, \eta ) &= -(f, v), & & \forall (v,\eta) \in V \times \Gamma , \\
\label{eq:weak-eq3} (\kappa^{-1} z, w) + ( p, \div w)  &= 0, & & \forall w \in W, \\
\label{eq:weak-eq4} (s_0 \dot{p}, q) + (\mc{A} ( \dot \sigma + \alpha \dot{p} \I), \alpha q \I) - (\div {z}, q) &= (g, q), & & \forall q \in Q .
\end{align} }
In \eqref{eq:weak-eq4} we used $(\tr \mc{A}^s \xi, q) = (\tr \mc{A} \xi, q) = (\mc{A} \xi, q \I)$ for a matrix $\xi$ and a scalar $q$. 

This variational formulation can be easily extended to problems with general boundary conditions. Suppose that boundary conditions are given as
\tblu{
\begin{align}
%\label{eq:bc1}  
%\begin{split}
p(t) &= p_0 (t) \quad \text{ on } \Gamma_p, &z(t) \cdot \bs{n} &= z_{\bs{n}} (t) \quad \text{ on } \Gamma_f, \\
%\label{eq:bc2}  
u(t) &= u_0 (t) \quad \text{ on } \Gamma_d, &\sigma(t) \bs{n} &= \sigma_{\bs{n}} (t) \quad \text{ on } \Gamma_t,
%\end{split}
\end{align} }
with 
\algns{
p_0 : J \times \Gamma_p \ra \R, \quad z_{\bs{n}} : J \times \Gamma_f \ra \R, \quad u_0 : J \times \Gamma_d, \quad \sigma_{\bs{n}} : J \times \Gamma_t \ra \R^n .
}
We define $\Sigma_{\pd}$ and $W_{\pd}$ as 
\begin{align} \label{eq:bdy-space}
\Sigma_{\pd} = \{ \tau \in \Sigma \;:\; \tau \bs{n} = 0 \text{ on } \Gamma_t \} , \quad W_{\pd} = \{ w \in W \;:\; w \cdot \bs{n} = 0 \text{ on } \Gamma_f \} .
\end{align}
The boundary conditions $\sigma_{\bs{n}}$ and $z_{\bs{n}}$ are imposed as essential boundary conditions and we can obtain a variational formulation 
%with vanishing $\sigma_{\bs{n}}$ and $z_{\bs{n}}$ 
with function spaces $\Sigma_{\pd}$ and $W_{\pd}$ \cite{Braess}. As a consequence, we have a variational formulation 
\algn{
\label{eq:weak-eq1bc}(\mathcal{A} (\sigma + \alpha p \I), \tau ) + (u, \div \tau) + (\gamma, \tau) &= \LRa{ u_0, \tau \bs{n}}_{\Gamma_d} , & & \forall \tau \in \Sigma_{\pd}, \\
\label{eq:weak-eq2bc}(\div \sigma, v) + (\sigma, \eta ) &= -(f, v), & & \forall (v,\eta) \in V \times \Gamma , \\
\label{eq:weak-eq3bc} (\kappa^{-1} z, w) + ( p, \div w)  &= \LRa{ p_0, w \cdot \bs{n}}_{\Gamma_p}, & & \forall w \in W_{\pd}, \\
\label{eq:weak-eq4bc}(s_0 \dot{p}, q) + (\mc{A} ( \dot \sigma + \alpha \dot{p} \I), \alpha q \I) - (\div {z}, q) &= (g, q), & & \forall q \in Q ,
}
where $\LRa{\cdot, \cdot}_{\Gamma_d}$ and $\LRa{\cdot, \cdot}_{\Gamma_p}$ are the integrals of $L^2$ inner products on the indicated parts of boundary with the $(n-1)$-dimensional Lebesgue measure.

\section{A priori error analysis} \label{sec:semi}
In this section we consider semidiscretization of \eqref{eq:weak-eq1}--\eqref{eq:weak-eq4} and prove a priori error estimates. The discrete counterpart of \eqref{eq:weak-eq1}--\eqref{eq:weak-eq4} is to seek 
\begin{align*}
(\sigma_h, p_h) \in C^1(J; \Sigma_h \times Q_h ) \quad \text{and} \quad  (u_h, \gamma_h, z_h) \in C^0(J; V_h \times \Gamma_h \times W_h), 
\end{align*}
such that 
\begin{align}
\label{eq:disc-eq1} (\mathcal{A} (\sigma_h + \alpha p_h \I), \tau) + (u_h, \div \tau) + (\gamma_h, \tau) &= 0 & & \forall \tau\in \Sigma_h, \\
\label{eq:disc-eq2} (\div \sigma_h, v) + (\sigma_h, \eta ) &= -(f, v) & & \forall (v, \eta) \in V_h \times \Gamma_h , \\
\label{eq:disc-eq3} (\kappa^{-1} z_h, w) + ( p_h, \div w)  &= 0 & & \forall w \in W_h, \\
\label{eq:disc-eq4} (s_0 \dot{p}_h, q) + (\mc{A} ( \dot \sigma_h + \alpha \dot{p}_h \I), \alpha q \I) - (\div z_h, q) &= (g, q) & & \forall q \in Q_h,
\end{align}
for appropriate finite element spaces $\Sigma_h$, $V_h$, $W_h$, $Q_h$, $\Gamma_h$ on a shape-regular mesh of $\Omega$. Note that \tred{this is a differential algebraic equation because the time derivative is involved only in the last equation, so existence and uniqueness of its solutions are not straightforward. However, we do not discuss existence and uniqueness of solutions of this semidiscrete problem here because they can be established with standard techniques in the theory of differential algebraic equations \cite{MR1027594}. Instead, we focus on illustrating details of the a priori error analysis of semidiscrete solutions.} At the end of this section we will discuss existence and uniqueness of fully discrete solutions for the problem, which is sufficient for practical computation.

%the purpose of semidiscrete error analysis is illustrating main ideas and procedures of error analysis without full technical details, and existence of solutions for the problem is not important for our a priori error analysis. 

In the rest of this paper we assume that $(\Sigma_h, V_h, \Gamma_h)$ is a stable mixed method for linear elasticity with weakly symmetric stress, and $(W_h, Q_h)$ is a stable mixed method for the mixed Poisson problem. Before we describe assumptions {\bf (S1)}--{\bf (S4)} including the stability conditions of finite elements, we restrict our interest on specific combinations of $(\Sigma_h, V_h, \Gamma_h)$ and $(W_h, Q_h)$ in order to have balanced convergence rates of unknowns. To state the conditions rigorously, let $\mathcal{O}(\Xi_h)$ be the best order of approximation of $\Xi_h$ in the $L^2$-norm for a function space $\Xi \subset L^2(\Omega; \X)$ and a finite element space $\Xi_h \subset \Xi$. For example, if $\Xi = L^2(\Omega)$ and $\Xi_h$ is the space of piecewise discontinuous polynomials of degree $\leq r$ on a shape-regular mesh, then $\mc{O}(\Xi_h) = r+1$. In this paper we always assume that $(\Sigma_h, V_h, \Gamma_h, W_h, Q_h)$ satisfies
\begin{align} \label{eq:conv-rate}
\min \{ \mathcal{O}(\Sigma_h), \mathcal{O}(\Gamma_h) \} = \mathcal{O}(V_h) = \mathcal{O}(W_h) = \mathcal{O}(Q_h)= r, \quad r \geq 1,
\end{align}
and 
\algns{
\mc{O}(V_h) = r', \quad (r' \geq 1) \quad r' = r \text{ or } r' = r-1 .
}
% \algns{
% \begin{cases}
% r' = r, &\text{ if } 
% %\mathcal{O}(V_h)  = r, &\quad r \geq 0, \\
% %\mathcal{O}(V_h) = r-1 , &\quad r \geq 1 .
% \end{cases}
% }
%As a consequence of our error analysis it will turn out that any combination of $(\Sigma_h, V_h, \Gamma_h)$ and $(W_h, Q_h)$ satisfying the above assumptions is possible. 
%For example, one can choose the Arnold--Falk--Winther element \cite{AFW07} for $(\Sigma_h, V_h, \Gamma_h)$ and the Raviart--Thomas--Nedelec element for .  
% \begin{align*}
% \inf_{\xi_h \in \Xi_h} \| \xi - \xi_h \|_0 \lesssim h^m \| \xi \|_m, \qquad 1 \leq m \leq \mathcal{O}(\Xi_h)
% \end{align*}
%holds. 
Here are the assumptions on finite elements: 
%In addition, we assume that there exists $c>0$ such that 
\begin{itemize}
\item[{\bf (S1)}] $\div \Sigma_h = V_h$ and $\div W_h = Q_h$.
\item[{\bf (S2)}] There exists bounded maps $\Pi_h^{\Sigma} : H^1(\Omega; \M) \ra \Sigma_h$, $\Pi_h^W : H^1(\Omega; \V) \ra Q_h$ such that 
\gats{
\div \Pi_h^{\Sigma} \tau = P_h^V \div \tau, \qquad \div \Pi_h^W w = P_h^Q \div w , \\
\| \tau - \Pi_h^\Sigma \tau \|_0 \lesssim h^m \| \tau \|_m, \qquad \| w - \Pi_h^W w \|_0 \lesssim h^m \| w \|_m , \qquad m \geq 1,
}
where $P_h^V$ and $P_h^Q$ are the $L^2$ projections into $V_h$ and $Q_h$, respectively.
\item[{\bf (S3)}] For any $q \in Q_h$, there exists $w \in Z_h$ satisfying
\begin{align*}
\div w = q, \qquad \| w \|_{\div} \lesssim \| q \|_0 .
\end{align*}
\item[{\bf (S4)}] For any $(v, \eta ) \in V_h \times \Gamma_h$, there exists $\tau \in \Sigma_h$ satisfying
\begin{align*}
\div \tau = v, \qquad (\tau, \eta' ) = ( \eta, \eta' ) \quad \forall \eta' \in \Gamma_h, \qquad \| \tau \|_{\div} \lesssim \| v \|_0 + \| \eta \|_0.
\end{align*}
\end{itemize}
There are a number of mixed finite elements satisfying these assumptions. For the mixed Poisson problems, all of the families of Raviart--Thomas, Brezzi--Douglas--Marini, N\'{e}d\'{e}lec elements \cite{BDM85,Nedelec80,Nedelec86,RT75} fulfill these assumptions. However, we will only consider the Raviart--Thomas elements (2D) and the N\'{e}d\'{e}lec 1st kind $H(\div)$ (3D) elements for $(W_h, Q_h)$ due to \eqref{eq:conv-rate}. For the mixed form of linear elasticity, there are many known elements satisfying these assumptions \cite{ABD84,AFW07,CGG10,GG10,Sten86,Sten88}. We refer to \cite{BFBook} and \cite{Lee14a} for surveys of the elements for mixed Poisson and mixed elasticity problems, respectively.

For the a priori error analysis of the problem, we need interpolation operators into finite element spaces. As previously defined $P_h^V$ and $P_h^Q$, we define $P_h^\Gamma : \Gamma \ra \Gamma_h$ as the $L^2$ projection into $\Gamma_h$. 
%\begin{align*}
%P_h^V, P_h^Q, P_h^{\Gamma} &: \text{ the orthogonal }L^2 \text{ projections into } V_h, Q_h, \Gamma_h .
%I_h^W &: \text{ the canonical Raviart--Thomas--N\'{e}d\'{e}lec interpolation}.
%P_h^{\Sigma} &: \text{ the weakly symmetric elliptic projection }
%\end{align*}
% It is well-known that $P_h^W$ and $P_h^Q$ satisfy a commuting diagram property
% \begin{align} \label{eq:RT-commute}
% P_h^Q \div = \div \Pi_h^W. 
% \end{align}
For $W_h$ let $I_h^W$ be the canonical Raviart--Thomas--N\'{e}d\'{e}lec interpolation. Then it is well-known that %For $\Sigma_h$ we will define interpolation operators $I_h^{\Sigma}$ in a sense of elliptic projection. 
% More precisely, for a given $z \in W$, we first consider the problem to seek $(\tilde{z}_h, \tilde{p}_h) \in W_h \times Q_h$ such that 
% \algns{
% (\kappa^{-1} \tilde{z}_h, w) + (\tilde{p}_h, \div w) &= (\kappa^{-1} z, w) & & \forall w \in W_h, \\
% (\div \tilde{z}_h, q) &= (\div z, q) & & \forall q \in Q_h .
% }
% If we define $I_h^W z := \tilde{z}_h$, then $I_h^W : W \ra W_h$ is a bounded linear map and the second equation yields
\gat{ \label{eq:poisson-elliptic1}
\div I_h^W w = P_h^Q \div w , \\
\label{eq:poisson-elliptic2}
\| w - I_h^W w \|_0 \lesssim h^m \| w \|_m, \quad 1 \leq m \leq r \qquad \text{for } w \in H^m(\Omega; \V)
}
hold. For $I_h^{\Sigma}$ we use the weakly symmetric elliptic projection introduced in \cite{Arnold-Lee}. To be self-contained, we describe its definition here. %The existence of $I_h^{\Sigma}$ explained below works for all elements satisfying the assumptions {\bf (S1)}--{\bf (S4)} \cite{Lee14a}. 
To define $I_h^{\Sigma} : \Sigma \ra \Sigma_h$ we consider a problem seeking $(\tilde{\sigma}_h, \tilde{u}_h, \tilde{\gamma}_h) \in \Sigma_h \times V_h \times \Gamma_h$ such that 
\tblu{
\begin{align}
\label{eq:ellip1} ( \tilde{\sigma}_h, \tau) + (\tilde{u}_h, \div \tau) + (\tilde{\gamma}_h, \tau) &= ( \sigma, \tau) , & & \forall \tau \in \Sigma_h, \\
\label{eq:ellip2} (\div \tilde{\sigma}_h, v) &= (\div \sigma, v), & & \forall v \in V_h, \\
\label{eq:ellip3} (\tilde{\sigma}_h, \eta ) &= (\sigma, \eta ), & & \forall \eta \in \Gamma_h,
\end{align} }
for $\sigma \in \Sigma$. This is the discrete counterpart of a mixed form of linear elasticity problem with weakly symmetric stress. Note that the continuous version of this linear elasticity problem has $(\sigma, 0, 0)$ as its solution. If we define $I_h^{\Sigma} \sigma$ by $\tilde{\sigma}_h$, then $I_h^{\Sigma}$ is a bounded linear map from $\Sigma$ to $\Sigma_h$. 
% Then, under the aforementioned assumptions of $(\Sigma_h, V_h, \Gamma_h)$, the error analysis of linear elasticity problem gives 
% \begin{align*}
% \| \sigma - \tilde{\sigma}_h \|_0 + \| 0 - \tilde{u}_h \|_0 + \| \gamma - \tilde{\gamma}_h \|_0 \leq c (\| \sigma - \Pi_h \sigma \|_0 + \| \gamma - \tilde{P}_h^{\Gamma} \gamma \|_0),
% \end{align*}
% in which $\tilde{P}_h^{\Gamma}$ is the $L^2$ projection into $\Gamma_h$. Now we define $P_h^{\Sigma} \sigma$ and $P_h^{\Gamma} \gamma$ as $\tilde{\sigma}_h$ and $\tilde{\gamma}_h$. We remark that both of $\sigma$ and $\gamma$ are used to define $P_h^{\Sigma} \sigma$ and $P_h^{\Gamma} \gamma$. Hence, strictly speaking, $P_h^{\Sigma} \sigma$ and $P_h^{\Gamma} \gamma$ are not mathematically rigorous notations but we will use them for convenience.
%This $P_h^{\Sigma}$ first appeared in [Arnold-Lee] and .  satisfying 
By \eqref{eq:ellip2}, the fact $\div \Sigma_h = V_h$, and \eqref{eq:ellip3}, we obtain 
\begin{align}
\label{eq:elliptic1} \div I_h^{\Sigma} \sigma = P_h^V \div \sigma \quad \text{and} \quad ( I_h^\Sigma \sigma, \eta ) = (\sigma, \eta ) \quad \forall \eta \in \Gamma_h .
\end{align}
An improved error estimate of mixed elasticity problems yields \cite{Arnold-Lee}
\begin{align}
\label{eq:elliptic2} \| I_h^{\Sigma} \tau - \tau \|_0 \lesssim h^m \| \tau \|_m, \quad 1 \leq m \leq r \qquad \text{if } \tau \in H^m(\Omega; \M).
\end{align}

\subsection{Error analysis}
Let $(\sigma, u, \gamma, z, p)$ and $(\sigma_h, u_h, \gamma_h, z_h, p_h)$ be solutions of \eqref{eq:weak-eq1}--\eqref{eq:weak-eq4} and \eqref{eq:disc-eq1}--\eqref{eq:disc-eq4}, respectively. We define the difference of these two solutions as error terms, i.e., $e_\sigma = \sigma - \sigma_h$,
and $e_u$, $e_\gamma$, $e_z$, $e_p$ are defined similarly. For the error analysis we decompose these error terms into two components using interpolation operators. More precisely, we set 
\begin{align*}
e_{\sigma} &= e_{\sigma}^I + e_{\sigma}^h := (\sigma - I_h^{\Sigma}\sigma) + (I_h^{\Sigma}\sigma - \sigma_h), \\
e_u &= e_u^I + e_u^h := (u - P_h^V u ) + (P_h^V u - u_h), \\
e_\gamma &= e_\gamma^I + e_\gamma^h := (\gamma - P_h^{\Gamma} \gamma ) + (P_h^{\Gamma} \gamma  - \gamma_h), \\
e_z &= e_z^I + e_z^h := (z - I_h^W z ) + (I_h^W z - z_h), \\
e_p &= e_p^I + e_p^h := (p - P_h^Q p ) + (P_h^Q p - p_h) .
\end{align*}
By well-known approximation properties of the interpolation operators $P_h^V$, $P_h^Q$, $P_h^{\Gamma}$, $I_h^W$ and by \eqref{eq:elliptic2}, we have 
\begin{align} 
\label{eq:semi-sigma-err} \| e_{\sigma}^I \|_{L^\infty L^2} &\lesssim h^m \| \sigma \|_{L^\infty H^m}, & & 1 \leq m \leq r, \\
\label{eq:semi-u-err} \| e_u^I \|_{L^\infty L^2} &\lesssim h^m \|  u \|_{L^\infty H^m}, & & 1 \leq m \leq r', \\
\label{eq:semi-z-err} \| e_z^I \|_{L^\infty L^2} &\lesssim h^m \|  z \|_{L^\infty H^m}, & & 1 \leq m \leq r, \\
\label{eq:semi-p-err} \| e_p^I \|_{L^\infty L^2} &\lesssim h^m \|  p \|_{L^\infty H^m}, & & 1 \leq m \leq r, \\
\label{eq:semi-gamma-err} \| e_\gamma^I \|_{L^\infty L^2} &\lesssim h^m \| \gamma \|_{L^\infty H^m}, & & 1 \leq m \leq r.
\end{align}
The following lemma will be useful in the error analysis.
\begin{lemma} \label{lemma:imp-gronwall}
Let $F, G, X :J \ra \R$ be continuous, nonnegative functions. Suppose that $X(t)$ is continuously differentiable and satisfies
\begin{align} \label{eq:Q-ineq1}
X(t)^2 \leq X(0)^2 + \int_0^{t} (F\tred{(s)} X \tred{(s)}+ G\tred{(s)}) \,ds.
\end{align}
for all $t \in J$. Then for $t \in J$, 
\begin{align} \label{eq:diff-ineq}
X(t) \leq X(0) + \max \left \{  2 \int_0^{t} F \tred{(s)} \,ds , \,\left ( 2\int_0^{t} G \tred{(s)} \,ds \right )^\half \right \}. 
\end{align}
\end{lemma}
\begin{proof} To reduce the problem, let us define a statement:
\bitem{
\item[{\bf (M)}] \eqref{eq:diff-ineq} holds when $X(t)$ is the maximum on the interval $[0,t]$
}
We first claim that, if {\bf (M)} is true, then \eqref{eq:diff-ineq} holds for any $t \in J$. To see this, assume that {\bf{(M)}} is true. For given $t_0 \in J$, if $X$ attains its maximum on $[0,t_0]$ at $t_0$, then there is nothing to prove, so suppose that $X$ attains its maximum value on the inverval $[0,t_0]$ at $\bar{t} \in [0, t_0]$, $\bar{t} < t_0$. Due to {\bf (M)} and the nonnegativity of $F$ and $G$, we have 
\algns{
X(t_0) < X(\bar{t}) &\leq X(0) + \max \left \{  2 \int_0^{\bar{t}} F \tred{(s)}\,ds , \,\left ( 2\int_0^{\bar{t}} G\tred{(s)} \,ds \right )^\half \right \} \\
&\leq X(0) + \max \left \{  2 \int_0^{t_0} F \tred{(s)}\,ds , \,\left ( 2\int_0^{t_0} G \tred{(s)}\,ds \right )^\half \right \},
}
so \eqref{eq:diff-ineq} holds for $t_0$ as well. Since $t_0 \in J$ is arbitrary, \eqref{eq:diff-ineq} for all $t \in J$.

%Suppose that $X(t)$ attains its maximum at $\bar t \in J$ and note that it suffices to prove \eqref{eq:diff-ineq} for $t = \bar t$. If $\bar t = 0$, then there is nothing to prove, so we assume $\bar t >0$. 
The above argument implies that it is enough to prove {\bf (M)}. Thus, we now show that \eqref{eq:diff-ineq} holds under the assumption that $X(t)$ is the maximum on the interval $[0,t]$. It is obvious that one of the following inequalities holds:
\begin{align} \label{eq:two-cases}
\int_0^{t} G\tred{(s)} \,ds  \leq \int_0^{t} F\tred{(s)} X\tred{(s)} \,ds, \qquad \int_0^{t} F\tred{(s)} X\tred{(s)} \,ds < \int_0^{t} G\tred{(s)} \,ds.
\end{align}
Suppose that the first inequality in \eqref{eq:two-cases} holds. Then \eqref{eq:Q-ineq1} gives
\begin{align*}
X(t)^2 \leq X(0)^2 + 2\int_0^{t} F\tred{(s)} X\tred{(s)} \,ds.
\end{align*}
If we divide both sides by $X(t)$, then we get
\begin{align*}
X(t) \leq X(0) + 2\int_0^{t} F\tred{(s)} \,ds , %\leq X(0) + 2\int_0^{T} F \,ds,
\end{align*}
because $X(t)$ is the maximum on $[0,t]$. Thus one case of \eqref{eq:diff-ineq} is proved. 

To complete the proof, assume that the second inequality in \eqref{eq:two-cases} is true. Then \eqref{eq:Q-ineq1} gives 
\begin{align*}
X({t})^2 \leq X(0)^2 + 2 \int_0^{t} G\tred{(s)} \,ds . %\leq X(0)^2 + 2 \int_0^{T} G \,ds.
\end{align*}
By taking square roots of both sides, and by the triangle inequality, we have
\begin{align*}
X({t}) \leq X(0) + \left( 2 \int_0^{t} G \tred{(s)}\,ds \right)^{\half},
\end{align*}
so the other case of \eqref{eq:diff-ineq} is proved. \qed
\end{proof}

% Let us introduce a lemma which will be used in the error analysis. %We refer to \cite{Lee14b} for its proof.
% \begin{lemma} \label{imp-gronwall}
%  Let $F, G, X, Y :J \ra \R$ be continuous, nonnegative functions. Suppose that $X(t)$ is continuously differentiable and satisfies
% \begin{align} \label{diff-ineq}
% \half \frac d{dt} X^2 + Y \leq FX + G,
% \end{align}
% for all $t \in J$. Then there exists $c>0$ independent of $t$ such that 
% \begin{align} \label{diff-ineq2}
% X^2(t) + \int_0^t Y ds  \leq c\left( X^2(0) +  \left( \int_0^t F ds \right)^2 + \int_0^t G ds  \right )
% \end{align}
% holds.
% \end{lemma}
We will use $\| \tau \|_{\mc{A}}$, $\| w \|_{\kappa^{-1}}$, $\| q \|_{s_0}$ to denote the quantities 
\algns{
(\mc{A} \tau , \tau)^{1/2}, \qquad (\kappa^{-1} w, w)^{1/2}, \qquad (s_0 q, q)^{1/2},
}
respectively. We also use $L_{\mc{A}}^2$ to denote the $L^2$ space with the norm $\| \cdot \|_{\mc{A}}$. 

Here is the main result of this section. 
\begin{theorem} \label{thm:semi-err} Let $(\sigma, u, \gamma, z, p)$ be an exact solution of \eqref{eq:weak-eq1}--\eqref{eq:weak-eq4} and assume that numerical initial data $\sigma_h(0)$, $p_h(0)$, $z_h(0)$ are given to satisfy
\algn{ \label{eq:ic}
\| \sigma(0) - \sigma_h(0), p(0) - p_h (0), z(0) - z_h(0) \|_0 \lesssim h^m, \qquad 1 \leq m \leq r. 
}
If $(\sigma_h, u_h, \gamma_h, z_h, p_h)$ is the solution of \eqref{eq:disc-eq1}--\eqref{eq:disc-eq4} with the numerical initial data, then, for $1 \leq m \leq r$, the following estimates hold:
\algn{
\label{eq:sigmap-estm} \| e_{\sigma}^h + \alpha e_p^h \I \|_{L^\infty L_\mc{A}^2} &\lesssim h^m \max \{ \| \sigma, p, \gamma \|_{W^{1,1} H^m}, \| z \|_{L^2 H^m} \}, \\
\label{eq:sigma-estm} \| e_{\sigma}^h \|_{L^\infty L^2} &\lesssim h^m ( \| \sigma, p, \gamma \|_{W^{1,2}H^m} + \| z \|_{W^{1,1}H^m} ), \\
\label{eq:ugamma-estm} \| e_u^h, e_{\gamma}^h \|_{L^\infty L^2} &\lesssim h^m \max \{\| \sigma, p, \gamma \|_{W^{1,1}H^m},  \| z \|_{L^2 H^m} \}, \\
\label{eq:z-estm} \| e_z^h \|_{L^\infty L^2} &\lesssim h^m ( \| \sigma, p, \gamma \|_{W^{1,2}H^m} + \| z \|_{W^{1,1}H^m}),  \\
\label{eq:p-estm} \| e_p^h \|_{L^\infty L^2} &\lesssim h^m ( \| \sigma, p, \gamma \|_{W^{1,2}H^m} + \| z \|_{W^{1,1}H^m}). 
}
Note that, for sufficiently regular solutions, $\| e_u^h \|_{L^\infty L^2} \lesssim h^r$ holds even though $r' = \mc{O}(V_h)$ is less than $r$. 
\end{theorem}
\begin{proof}
%By the triangle inequality and \eqref{eq:semi-sigma-err}--\eqref{eq:semi-gamma-err} it is enough to estimate $e^h$-terms.
% \begin{align} \label{eq:semi-eh-err}
% \| e_{\sigma}^h, e_v^h, e_z^h, e_p^h, e_\gamma^h \|_{L^\infty L^2} \lesssim h^m \| \sigma, v, z, p, \gamma \|_{W^{1,1} H^m}, \qquad 1 \leq m \leq r.
% \end{align}
We begin the error analysis by taking differences of the variational form and the semidiscrete equations, which results in
\begin{align}
\label{eq:err-eq1} (\mathcal{A} (e_\sigma + \alpha e_p \I), \tau) + (e_u, \div \tau) + (e_\gamma, \tau) &= 0, & & \forall \tau\in \Sigma_h, \\
\label{eq:err-eq2} (\div e_\sigma, v) + (e_\sigma, \eta ) &= 0, & & \forall (v,\eta) \in V_h \times \Gamma_h , \\
\label{eq:err-eq3} (\kappa^{-1} e_z, w) + ( e_p, \div w)  &= 0, & & \forall w \in W_h, \\
\label{eq:err-eq4} (s_0 \dot{e}_p, q) + (\mc{A} ( \dot{e}_\sigma + \alpha \dot{e}_p \I), \alpha q \I) - (\div e_z, q) &= 0, & & \forall q \in Q_h .
\end{align}
Note that {\bf (S1)} and the definitions of $P_h^Q$, $P_h^V$ yield 
\begin{align*}
(e_p^I, \div w) = 0 \quad \forall w \in W_h  \quad \text{ and } \quad  (e_u^I, \div \tau) = 0 \quad \forall \tau \in \Sigma_h.
\end{align*}
In addition, the commuting diagram property of $I_h^W$ in {\bf (S2)} and the properties of $I_h^\Sigma$ in \eqref{eq:elliptic1} yield
%and properties \eqref{eq:RT-commute} and \eqref{eq:elliptic1} yield 
\begin{align*}
(\div e_z^I, q) = 0 \quad \forall q \in Q_h \quad \text{ and }\quad  (\div e_\sigma^I, v) = (e_\sigma^I, \eta ) = 0 \quad \forall (v, \eta ) \in V_h \times \Gamma_h.
\end{align*}
Considering these cancellations and using the error decompositions with $e^I$- and $e^h$-terms, the system \eqref{eq:err-eq1}--\eqref{eq:err-eq4} can be rewritten as 
\begin{align}
\label{eq:err-eq1n} (\mathcal{A} (e_\sigma^h + \alpha {e}_p^h \I), \tau) + ({e}_u^h, \div \tau) + ({e}_\gamma^h, \tau) &=  -(\mathcal{A} (e_\sigma^I + \alpha e_p^I \I) , \tau) - (e_\gamma^I, \tau) , \\
\label{eq:err-eq2n} (\div e_\sigma^h, v) + (e_\sigma^h, \eta )&= 0, \\
\label{eq:err-eq3n} (\kappa^{-1} e_z^h, w) + ( e_p^h, \div w)  &= - (\kappa^{-1} e_z^I, w), \\
\label{eq:err-eq4n} (s_0 \dot{e}_p^h, q) + (\mc{A} ( \dot{e}_\sigma^h + \alpha \dot{e}_p^h \I), \alpha q \I) - (\div e_z^h, q) &= - (s_0 \dot{e}_p^I, q) \\
\notag & \quad - (\mc{A} ( \dot{e}_\sigma^I + \alpha \dot{e}_p^I \I), \alpha q \I), 
\end{align}
for any $(\tau, v, \eta, w, q) \in (\Sigma_h, V_h, \Gamma_h, W_h, Q_h)$. 

We remark that \eqref{eq:ic} implies that $\| e_\sigma^h(0) \|_0 \lesssim h^m$ because 
\algns{
\| e_\sigma^h (0) \|_0 &= \| I_h^{\Sigma} \sigma(0) - \sigma_h (0) \|_0 \\
&\leq \| I_h^{\Sigma} \sigma(0) - \sigma(0) \|_0 + \| \sigma(0) - \sigma_h(0) \|_0 \\
&\lesssim h^m \| \sigma(0) \|_m .
}
We obtain $\| e_p^h(0), e_z^h(0) \|_0 \lesssim h^m$ with similar arguments. 

{\bf Estimate of $\| e_\sigma^h + \alpha e_p^h \I \|_{\mc{A}}$} : We first show  
\mltln{ \label{eq:sp-estm}
\| e_\sigma^h + \alpha e_p^h \I \|_{L^\infty L_\mc{A}^2} \lesssim h^m ( \| \sigma(0), p(0) \|_m + \max \{ \| \sigma, p, \gamma \|_{L^2 H^m} , \| z \|_{L^2 H^m} \}.
}
For its proof, taking the time derivative of \eqref{eq:err-eq2n}, choosing $\tau = e_\sigma^h$, $v = \dot{e}_u^h$, $w = e_z^h$, $q = e_p^h$, $\eta = -\dot{e}_\gamma^h$, and adding all the equations \eqref{eq:err-eq1n}--\eqref{eq:err-eq4n}, we have 
\begin{multline*}
\half \frac{d}{dt} (\| e_\sigma^h + \alpha e_p^h \I \|_{\mc{A}}^2 + \| e_p^h \|_{s_0}^2) + \|  e_z^h \|_{\kappa^{-1}}^2 \\
= - (\kappa^{-1} e_z^I, e_z^h) - (\mc{A} (\dot{e}_\sigma^I + \alpha \dot{e}_p^I \I ), e_\sigma^h + \alpha e_p^h \I) - (s_0 \dot{e}_p^I, e_p^h) - (\dot{e}_\gamma^I, e_\sigma^h) .
\end{multline*}
The Cauchy--Schwarz and arithmetic-geometric mean inequalities give
\begin{multline*}
\half \frac{d}{dt} (\| e_\sigma^h + \alpha e_p^h \I \|_{\mc{A}}^2 + \| e_p^h \|_{s_0}^2) + \half \| e_z^h \|_{\kappa^{-1}}^2 \\
\leq \half \| e_z^I \|_{\kappa^{-1}}^2 + \| \dot{e}_\sigma^I + \alpha \dot{e}_p^I \I  + \dot{e}_{\gamma}^I \|_{\mc{A}} \| e_\sigma^h + \alpha e_p^h \I \|_{\mc{A}} + \| \dot{e}_p^I \|_{s_0} \| e_p^h \|_{s_0} .
\end{multline*}
In this inequality we assumed that $(\mc{A} \tau, \dot{e}_{\gamma}^I) = ( \tau, \dot{e}_{\gamma}^I)$ holds for simplicity but this equality holds with a positive constant in general, which is used in the extension of $\mc{A}^s$ to $\mc{A}$ for skew-symmetric tensors. Ignoring the nonnegative term $\| e_z^h \|_{\kappa^{-1}}^2 /2$ in the above and applying Lemma \ref{lemma:imp-gronwall} with
\algns{
X &= \LRp{ \| e_\sigma^h + \alpha e_p^h \I \|_{\mc{A}} }^2 + \| e_p^h \|_{s_0}^2)^{1/2}, \\ %& Y &= \half \| e_z^h \|_{\kappa^{-1}}^2, \\
F &= \LRp{ \| \dot{e}_\sigma^I + \alpha \dot{e}_p^I \I + \dot{e}_{\gamma}^I \|_{\mc{A}}^2 + \| \dot{e}_p^I \|_{s_0}^2 }^\half, \\
G &= \half \| e_z^I \|_{\kappa^{-1}}^2,
}
we can obtain 
\algns{
\| {e}_\sigma^h + \alpha {e}_p^h \I  \|_{L^\infty L_\mc{A}^2} %+ \| e_z^h \|_{L^2 L^2} 
&\lesssim \| {e}_\sigma^h (0) + {e}_p^h (0) \I \|_0 + \max \{ \| \dot{e}_\sigma^I, \dot{e}_p^I, \dot{e}_{\gamma}^I \|_{L^1 L^2}, \| e_z^I \|_{L^2 L^2} \} \\
&\lesssim h^m ( \| {\sigma} (0), {p} (0) \|_m + \max \{ \| \sigma, p, \gamma \|_{W^{1,1} H^m} , \| z \|_{L^2 H^m} \}) \\
&\lesssim h^m \max \{ \| \sigma, p, \gamma \|_{W^{1,1} H^m} , \| z \|_{L^2 H^m} \} ,
}
where the last inequality is due to the Sobolev embedding $W^{1,1}H^m \subset L^\infty H^m $.

\noindent {\bf Estimates of $\| e_u^h, e_\gamma^h \|_{L^\infty L^2}$ } : By the inf-sup condition {\bf (S4)} there exists $\tau \in \Sigma_h$ such that $\div \tau = e_u^h$, $(\tau, \eta ) = (e_\gamma^h, \eta )$ for all $\eta \in \Gamma_h$, and $\| \tau \|_{\div}^2 \lesssim \| e_u^h \|_0^2 + \| e_{\gamma}^h \|_0^2$. Taking this $\tau$ in \tblu{\eqref{eq:err-eq1n}} and applying the triangle and Cauchy--Schwarz inequalities, we get
\begin{align*}
\| e_u^h \|_0^2 + \| e_\gamma^h \|_0^2 &= -(\mc{A} (e_\sigma + \alpha e_p \I + e_\gamma^I), \tau) \\
&\leq (\| e_\sigma^I + \alpha e_p^I \I + e_\gamma^I \|_{\mc{A}} + \| e_\sigma^h + \alpha e_p^h \I \|_{\mc{A}}) \| \tau \|_{\mc{A}} \\
&\lesssim (h^m \| \sigma, p, \gamma \|_{L^\infty H^m} + \| e_\sigma^h + \alpha e_p^h \I \|_{\mc{A}})  (\| e_u^h \|_0^2 + \| e_{\gamma}^h \|_0^2)^\half ,
\end{align*}
where we used $\| \tau \|_{\mc{A}} \lesssim \| \tau \|_{\div} \lesssim \LRp{\| e_u^h \|_0^2 + \| e_{\gamma}^h \|_0^2 }^{1/2}$ in the last inequality. Since $\| \sigma, p, \gamma \|_{L^\infty H^m} \lesssim \| \sigma, p, \gamma \|_{W^{1,1}H^m}$ by the Sobolev embedding, the above inequality and the estimate of $\| e_\sigma^h + \alpha e_p^h \I \|_{\mc{A}}$ yield 
\algns{
\| e_u^h,  e_\gamma^h \|_{L^\infty L^2} \lesssim h^m \max \{\| \sigma, p, \gamma \|_{W^{1,1}H^m},  \| z \|_{L^2 H^m} \} .
}
{\bf Estimate of $\| e_z^h \|_{L^\infty L^2}$} : Taking time derivatives of \tblu{\eqref{eq:err-eq1n}--\eqref{eq:err-eq3n}}, choosing $\tau = \dot{e}_\sigma^h$, \tblu{$v = -\dot{e}_u^h$}, $w = e_z^h$, $q = \dot{e}_p^h$, \tblu{$\eta = - \dot{e}_\gamma^h$}, and adding all the equations, we have
\begin{multline*}
\half \frac{d}{dt} (\kappa^{-1} e_z^h, e_z^h) + \| \dot{e}_\sigma^h + \alpha \dot{e}_p^h \I \|_{\mc{A}}^2 + \| \dot{e}_p^h \|_{s_0}^2 \\
= - (\kappa^{-1} \dot{e}_z^I, e_z^h) - (\mc{A} (\dot{e}_\sigma^I + \alpha \dot{e}_p^I \I + \dot{e}_\gamma^I ), \dot{e}_\sigma^h + \alpha \dot{e}_p^h) - (s_0 \dot{e}_p^I, \dot{e}_p^h).
\end{multline*}
By the Cauchy--Schwarz and Young's inequalities we have
\begin{align*}
\half \frac{d}{dt} \|  e_z^h \|_{\kappa^{-1}}^2 \leq  \| \dot{e}_z^I \|_{\kappa^{-1}} \| e_z^h \|_{\kappa^{-1}} + \| \dot{e}_\sigma^I + \alpha \dot{e}_p^I \I + \dot{e}_\gamma^I \|_{\mc{A}}^2 + \| \dot{e}_p^I \|_{s_0}^2.
\end{align*}
If we apply Lemma \ref{lemma:imp-gronwall} to this inequality with % and the coercivity of $\kappa^{-1}$ yield
\algns{
X = \| e_z^h \|_{\kappa^{-1}}, \qquad F = \| \dot{e}_z^I \|_{\kappa^{-1}}, \qquad G = \| \dot{e}_\sigma^I + \alpha \dot{e}_p^I \I + \dot{e}_\gamma^I \|_{\mc{A}}^2 + \| \dot{e}_p^I \|_{s_0}^2,
}
then the coercivity of $\kappa^{-1}$, boundedness of $\mc{A}$, and the Sobolev embedding will lead to 
\begin{align*}
\| e_z^h \|_{L^\infty L^2} 
&\lesssim \| e_z^h(0) \|_0 + \| \dot{e}_z^I \|_{L^1 L^2} + \| \dot{e}_\sigma^I, \dot{e}_p^I, \dot{e}_\gamma^I \|_{L^2 L^2} \\
&\lesssim h^m (\| \sigma, p, \gamma \|_{W^{1,2}H^m} + \| z \|_{W^{1,1}H^m}).
\end{align*}
{\bf Estimate of $\| e_p^h \|_{L^\infty L^2}$} : By the inf-sup condition in {\bf (S3)} for $(W_h, Q_h)$, there exists $w \in W_h$ such that $\div w = e_p^h$ and $\| w \|_{\div} \lesssim \| e_p^h \|_0$. Choosing this $w$ in \eqref{eq:err-eq3n}, we get 
\begin{align*}
\| e_p^h \|_0^2 = (e_z^I + e_z^h, w)_{\kappa^{-1}} \lesssim (\| e_z^I \|_0 + \| e_z^h \|_0) \| e_p^h \|_0,
\end{align*}
so the estimate of $e_z^h$ and \eqref{eq:semi-z-err} give 
\algns{
\| e_p^h \|_{L^\infty L^2} \lesssim h^m (\| \sigma, p, \gamma \|_{W^{1,2}H^m} + \| z \|_{W^{1,1}H^m}) .
}
{\bf Estimate of $\| e_{\sigma}^h \|_{L^\infty L^2}$} : Combining the estimates of $\| e_\sigma^h + \alpha e_p^h \I \|_{\mc{A}}$ and $\| e_p^h \|_{L^\infty L^2}$, the triangle inequality, and coercivity of $\mc{A}$, we have 
\algns{
\| e_{\sigma}^h \|_{L^\infty L^2} \lesssim h^m ( \| \sigma, p, \gamma \|_{W^{1,2}H^m} + \| z \|_{W^{1,1}H^m}) .
}
The proof is completed. \qed
\end{proof}
\begin{rmk} The quantity \tblu{$(\sigma + \alpha p \I)$} has a physical meaning as the \tblu{elastic} stress tensor in the saturated poroelastic medium. The estimate of $\| e_{\sigma}^h + \alpha e_p^h \I \|_{L^\infty L_{\mc{A}}^2}$ plays a key role that any combinations of two mixed methods $(\Sigma_h, V_h, \Gamma_h)$ and $(W_h, Q_h)$ are available because this estimate holds without any requirements on finite elements. In \tblu{this} estimate, \tblu{an} estimate of $\| e_p^h \|_{L^\infty L^2}$ can be obtained when $s_0$ is uniformly positive, and an estimate of $\| e_z^h \|_{L^2 L^2}$ can be also obtained.
\end{rmk}
\begin{rmk}
The argument in the estimate of $\| e_z^h \|_{L^\infty L^2}$ can be also used for other formulations that the fluid flux and the pore pressure are present as unknowns \cite{MR2327966,NUM:NUM21775}.
\end{rmk}

If we combine \eqref{eq:semi-sigma-err}--\eqref{eq:semi-gamma-err} and the result of Theorem~\ref{thm:semi-err}, then we have the following results. 
\begin{cor} \label{thm:full-err} Suppose that $(\sigma, u, \gamma, z, p)$ and $(\sigma_h, u_h, \gamma_h, z_h, p_h)$ are as in Theorem~\ref{thm:semi-err} with same assumptions. Then 
we have 
\algns{
\| e_{\sigma} + \alpha e_p \I \|_{L^\infty L_\mc{A}^2} &\lesssim h^m \max \{ \| \sigma, p, \gamma \|_{W^{1,1} H^m}, \| z \|_{L^2 H^m} \}, & & 1 \leq m \leq r, \\
\| e_{\sigma} \|_{L^\infty L^2} &\lesssim h^m ( \| \sigma, p, \gamma \|_{W^{1,2}H^m} + \| z \|_{W^{1,1}H^m} ), & & 1 \leq m \leq r, \\
\| e_u \|_{L^\infty L^2} &\lesssim h^m \max \{\| \sigma, p, \gamma \|_{W^{1,1}H^m},  \| z \|_{L^2 H^m} \}, & & 1 \leq m \leq r', \\
\| e_{\gamma} \|_{L^\infty L^2} &\lesssim h^m \max \{\| \sigma, p, \gamma \|_{W^{1,1}H^m},  \| z \|_{L^2 H^m} \}, & & 1 \leq m \leq r, \\
\| e_z \|_{L^\infty L^2} &\lesssim h^m ( \| \sigma, p, \gamma \|_{W^{1,2}H^m} + \| z \|_{W^{1,1}H^m}), & & 1 \leq m \leq r, \\
\| e_p \|_{L^\infty L^2} &\lesssim h^m ( \| \sigma, p, \gamma \|_{W^{1,2}H^m} + \| z \|_{W^{1,1}H^m}), & & 1 \leq m \leq r. 
}
\end{cor}

\subsection{Robustness of error estimates for nearly incompressible materials}

For isotropic elastic porous media, the elasticity tensor $\mathcal{C}$ has the form
\begin{align*} %\label{eq:homog-elasticity}
\mathcal{C} \tau = 2 \mu \tau + \lambda (\tr \tau) \I, \qquad \tau \in L^2(\Omega; \Bbb{S}), %\R_{\sym}^{n \times n}),
\end{align*}
where the constants $\mu, \lambda >0$ are Lam\'{e} coefficients, and 
\begin{align} \label{eq:A-form}
\mc{C}^{-1} = \mathcal{A}^s \tau = \frac{1}{2 \mu} \left( \tau - \frac{\lambda}{2 \mu + n \lambda} (\tr \tau) \I \right) .
\end{align}
Throughout this subsection, we assume that $\mc{A}^s$ has this form and $\mathcal{A}$ is the extension of $\mc{A}^s$ to $L^2(\Omega; \M)$ as before. One can see that the coercivity of $\mc{A}$ on $L^2(\Omega; \M)$ is not uniform in $\lambda$. In other words,
\algns{
c_{\lambda} \| \tau \|_0^2 \leq \| \tau \|_{\mc{A}}^2, \qquad \forall \tau \in L^2(\Omega; \M)
}
holds with a constant $c_{\lambda} >0$ but $c_{\lambda} \ra 0$ as $\lambda \ra + \infty$. It means that error estimates obtained using coercivity of $\mc{A}$ may have error bounds growing unboundedly as $\lambda \ra +\infty$. The purpose of this subsection is to show that the error bounds in the previous subsection are uniform for arbitrarily large $\lambda$. 

In the proof of Theorem~\ref{thm:semi-err} we can see that many error estimates rely on the estimate of $\| e_\sigma^h + \alpha e_p^h \I \|_{\mc{A}}$ and boundedness of the bilinear form $\mc{A}$. It is an important observation that estimates utilizing the $\mc{A}$-weighted norm and boundedness of $\mc{A}$ in \eqref{eq:A-form} are uniform as $\lambda \ra +\infty$.
%, so the error estimates obtained from those are robust for arbitrarily large $\lambda$. 
Thus, the only error estimate requiring coercivity of $\mc{A}$ is the error estimate of $\| e_\sigma^h \|_{L^\infty L^2}$. 

%In this subsection, we will consider both of $\Gamma_d = \pd \Omega$ and $\Gamma_d \not = \pd \Omega$ cases because they need slightly different arguments. 
\tred{Before we prove the main result, we need preliminary results. For a tensor $\tau$ we use $\tau^D$ to denote the deviatoric part of $\tau$, i.e.,
\begin{align*}
\tau^D := \tau - \frac 1n (\tr \tau) \I. 
\end{align*}
We define $H_{\Gamma}^1(\Omega)$ as
\algn{ \label{eq:H1Gamma}
H_{\Gamma}^1(\Omega) = 
%\begin{cases} 
%H_0^1(\Omega; \V) \\
\{ \phi \in H^1(\Omega; \V) \,:\, \phi |_{\Gamma_d} = 0 \},
%\end{cases}
}
where $\phi |_{\Gamma_d}$ is the trace of $\phi$ on $\Gamma_d \subset \pd \Omega$. Recall the definitions of $\Sigma$ in \eqref{eq:svgamma-space} and $\Sigma_{\pd}$ in \eqref{eq:bdy-space} depending on boundary conditions $\Gamma_d = \pd \Omega$ and $\Gamma_d \not = \pd \Omega$.
\begin{lemma}  \label{thm:dev-lemma} Suppose that 
\algn{ \label{eq:tau-dev-assume}
\tau \in \Sigma, \quad \int_{\Omega} \tr \tau \, dx = 0, \qquad \text{or } \qquad \tau \in \Sigma_{\pd} .
}
Then 
\algn{ \label{eq:tau-dev-estm}
\| \tau \|_0 \lesssim \| \tau^D \|_0 + \sup_{\phi \in H_{\Gamma}^1(\Omega)} \frac{ (\div \tau, \phi) }{\| \phi \|_1} .
}
\end{lemma}
}
\begin{proof}
\tred{ 
Since $(\tau^D, (\tr \tau) \I) =0$, $\| \tau \|_0 \lesssim \| \tau^D \|_0 + \| \tr \tau \|_0$ holds and it suffices to show that $\| \tr \tau \|_0$ is bounded by the right-hand side of \eqref{eq:tau-dev-estm}. Due to the assumption \eqref{eq:tau-dev-assume}, there exists $\phi \in H_{\Gamma}^1(\Omega)$ such that $\div^* \phi = \tr \tau$ and $\| \phi \|_1 \lesssim \| \tr \tau \|_0$, where $\div^*$ means the vertical divergence (cf. \cite{girault-raviart-NSE}). Then an algebraic identity and the integration by parts give 
\algns{ 
\| \tr \tau \|_0^2 = (\tr \tau , \div^* \phi) &= n(\tau , \grad \phi) - n(\tau^D, \grad \phi) \\
&= - n(\div \tau, \phi) - n(\tau^D, \grad \phi) .
}
We get the desired result from the Cauchy--Schwarz inequality, $\| \phi \|_1 \lesssim \| \tr \tau \|_0$, and dividing both sides by $\| \phi \|_1$. \qed
}
\end{proof}
\tred{
The proof of Lemma~\ref{thm:dev-lemma} is completely analogous to the proof of Lemma~3.1 in \cite{ADG84}, with a simple modification for general boundary conditions. A similar result was proved in \cite{MR2004196} with a Helmholtz decomposition. 
}

\begin{theorem} \label{thm:locking-free} Suppose that $(\sigma, u, \gamma, z, p)$ and $(\sigma_h, u_h, \gamma_h, z_h, p_h)$ are exact and discrete solutions of \eqref{eq:weak-eq1}--\eqref{eq:weak-eq4} and \eqref{eq:disc-eq1}--\eqref{eq:disc-eq4} with the assumptions as in Theorem~\ref{thm:semi-err}. We also assume that $\mc{A}$ is the extension of \eqref{eq:A-form}, and $\alpha$ is piecewise constant on each mesh element. Then 
\begin{align*}
\| \sigma - \sigma_h \|_0 \leq ch^m, \qquad 1 \leq m \leq r,
\end{align*}
holds with a constant $c$, which is uniformly bounded for arbitrarily large $\lambda$. Moreover, the same result holds for a solution of \eqref{eq:weak-eq1bc}--\eqref{eq:weak-eq4bc} (with $\Gamma_d \not = \pd \Omega$) and its discrete counterpart. 
\end{theorem}
\begin{proof} 
Since $I_h^{\Sigma}$ is independent of $\lambda$, $\| e_\sigma^I \|_0 \lesssim h^m$ holds with a constant independent of $\lambda$. By the triangle inequality, it is enough to have an estimate $\| e_\sigma^h \|_0$ with an upper bound uniformly bounded for arbitrarily large $\lambda$. 

We first consider the case with $\Gamma_d = \pd \Omega$. Since $\Sigma_h$ does not have any essential boundary condition on a part of $\pd \Omega$, we can take $\tau = \I \in \Sigma_h$ as a test function in error analysis. By taking $\tau = \I$ in \eqref{eq:ellip1}, we also have 
\begin{align} \label{eq:tr-mean}
\int_{\Omega} \tr (\sigma - I_h^{\Sigma} \sigma) \,dx = \int_{\Omega} \tr e_\sigma^I \,dx =  0 .
\end{align}
Taking $\tau = \I$ in \eqref{eq:err-eq1}, we have 
\algn{ \label{eq:tauI-eq}
(\mc{A} (e_\sigma + \alpha e_p \I), \I) = \frac{1}{2 \mu + n \lambda} \int_{\Omega} (\tr e_{\sigma} + \alpha n e_p )dx = 0. 
}
Let $| \Omega |$ be the $n$-dimensional Lebesgue measure of $\Omega$ and define 
\begin{align*}
%| \Omega | := \int_{\Omega} 1 \,dx, \qquad 
\beta := \frac{1}{n | \Omega | } \int_{\Omega} \tr e_\sigma \,dx .
\end{align*}
%and it will be useful when we show robustness of error estimates in nearly incompressible materials later. 
Taking $\tau = \I$ in \eqref{eq:err-eq1}, we have 
\begin{align} \label{eq:beta}
n | \Omega | \beta = \int_{\Omega} \tr e_\sigma \,dx =  \int_{\Omega} \tr e_\sigma^h \,dx = - n \int_{\Omega} \alpha e_p \,dx = - n \int_{\Omega} \alpha e_p^h \,dx ,
\end{align}
where the second, third, and fourth equalities are due to \eqref{eq:tr-mean}, \eqref{eq:tauI-eq}, and the facts that $\alpha$ is piecewise constant and $e_p^I$ is mean-value zero on each mesh element. Letting $\tilde{e}_\sigma^h = e_\sigma^h - \beta \I$ we have 
\begin{align} \label{eq:tr0}
\int_{\Omega} \tr \tilde{e}_\sigma^h \,dx = 0.
\end{align}
The triangle inequality gives 
\begin{align*}
\| e_\sigma^h \|_0 \leq \| \tilde{e}_\sigma^h \|_0 + |\beta | \| \I \|_0 = \| \tilde{e}_\sigma^h \|_0 + \sqrt{n} | \beta | | \Omega |^{\half}.
\end{align*}
By \eqref{eq:beta} and H\"older inequality, we have  
\begin{align*}
\sqrt{n} | \beta | | \Omega |^{\half} = \tblu{\sqrt{n} | \Omega |^{-\half} \left | \int_{\Omega} \alpha e_p^h \,dx \right| \lesssim \| e_p^h \|_0}.
\end{align*}
Recall that the estimate of $\| e_p^h \|_0$ is uniform in $\lambda$, so the above two inequalities imply that, to have a $\lambda$-uniform estimate of $\| e_\sigma^h \|_0$, we only need a $\lambda$-uniform estimate of $\| \tilde{e}_\sigma^h \|_0$. 
%To obtain such an estimate of $\| \tilde{e}_\sigma^h \|_0$ we use results in \cite{ADG84}. 
If we use \eqref{eq:tr0} and the fact $\div e_\sigma^h = \div \tilde{e}_\sigma^h = 0$ by definition of $\tilde{e}_\sigma^h$, then \eqref{eq:tau-dev-estm} gives \tred{$\| \tilde{e}_\sigma^h \|_0 \lesssim \| (\tilde{e}_\sigma^h)^D \|_0$}. 
In addition, it is easy to see $\| (\tilde{e}_\sigma^h)^D \|_0 \lesssim \| \tilde{e}_\sigma^h \|_{\mc{A}}$ from the definition of $\mc{A}$ (cf. \cite[Lemma~3.2]{ADG84}), so it is sufficient to estimate $\| \tilde{e}_\sigma^h \|_{\mc{A}}$. A direct computation using the form of $\mc{A}$ in \eqref{eq:A-form} gives
\begin{align*}
\| \tilde{e}_\sigma^h \|_{\mc{A}}^2 &= \left ( \mc{A} ( e_\sigma^h - \beta \I), e_\sigma^h - \beta \I \right )  \\
&= (\mc{A} e_\sigma^h, e_\sigma^h) - 2 \beta (\mc{A} \I,  e_\sigma^h) + \beta^2 (\mc{A} \I, \I) \\
&= \| e_\sigma^h \|_{\mc{A}}^2 - \frac{2 \beta }{2 \mu + n \lambda} \int_{\Omega} \tr e_\sigma^h \,dx + \frac{n \beta^2 }{2 \mu + n \lambda} | \Omega |  \quad (\text{by }\eqref{eq:A-form}) \\
&= \| e_\sigma^h \|_{\mc{A}}^2 - \frac{n \beta^2 }{2 \mu + n \lambda} | \Omega | \qquad  (\text{by }\eqref{eq:beta}) \\
&\leq \| e_\sigma^h \|_{\mc{A}}^2 .
\end{align*}
As we already have an estimate of $\| e_\sigma^h \|_{\mc{A}}$ which is uniform in $\lambda$, we obtain an estimate of $\| \tilde{e}_{\sigma}^h \|_0$ uniform in $\lambda$ as well. 

\tred{
If $\Gamma_d \not = \pd \Omega$, then we can apply \eqref{eq:tau-dev-estm} directly to $e_{\sigma}$ and have 
\algn{ \label{eq:esigma-estm}
\| e_{\sigma} \|_0 \lesssim \| (e_{\sigma})^D \|_0 + \sup_{\phi \in H_{\Gamma}^1 (\Omega)} \frac{ (\div e_{\sigma}, \phi)}{\| \phi \|_1} .
}
Then it is enough to show that the first and second terms in the above are bounded by $\| e_{\sigma}^I \|_0$ and $\| e_{\sigma}^h \|_{\mc{A}}$. The first term $\| (e_{\sigma})^D \|_0$ is easily estimated by 
\algns{ 
\| (e_{\sigma})^D \|_0 \leq \| (e_{\sigma}^I)^D \|_0 + \| (e_{\sigma}^h)^D \|_0 \lesssim \| e_{\sigma}^I \|_0 + \| e_{\sigma}^h \|_{\mc{A}} .
}
}
\tred{ 
To estimate the second term, note that \eqref{eq:err-eq2} with $\eta = 0$ implies that $\div e_{\sigma} = \div \sigma - P_h^V \div \sigma$, so 
\algns{ 
(\div e_{\sigma}, \phi) = (\div (\sigma - I_h^{\Sigma} \sigma) , \phi) = - (\sigma - I_h^{\Sigma} \sigma, \grad \phi) = - (e_{\sigma}^I, \grad \phi). 
}
Thus, the second term in \eqref{eq:esigma-estm} is bounded by $\| e_{\sigma}^I \|_0$. \qed
}
\end{proof}

\subsection{Superconvergence and a local post-processing}
We show that superconvergence of the displacement error is available for certain choices of $(\Sigma_h, V_h, \Gamma_h, W_h, Q_h)$. We also show that a new numerical displacement with higher order accuracy can be obtained from a local postprocessing of the superconvergent numerical displacement.

Suppose that $(\Sigma_h, V_h, \Gamma_h, W_h, Q_h)$ are elements such that 
\algns{
\mc{O}(\Sigma_h) = \mc{O}(\Gamma_h) = \mc{O}(W_h) = \mc{O}(Q_h) = r, \qquad \mc{O}(V_h) = r-1
}
with $r \geq 2$. Then we have obtained 
\begin{align} \label{eq:eh-super}
\| e_{\sigma}^h, e_u^h, e_z^h, e_p^h, e_{\gamma}^h \|_0 \lesssim h^r
\end{align}
in Theorem~\ref{thm:semi-err} assuming that exact solutions are sufficiently regular. 
% By the triangle inequality, the above implies that 
% \begin{align*}
% \| e_{\sigma}, e_z, e_p, e_{\gamma} \|_0 \lesssim h^r.
% \end{align*}
% In other words, errors of all the unknowns other than $u$ have optimal order of convergence in spite of the low approximability of $V_h$. 
In particular, \eqref{eq:eh-super} implies that the convergence rate of $\| e_u^h \|_0$ is higher than the best approximation order of $V_h$. This allows us to use a local post-processing to obtain $u_h^*$, a numerical solution of $u$ with better accuracy. 

We describe the local post-processing here. Let $V_h^*$ be the space of polynomials with $1$ degree higher than $V_h$, $V_{h,0} \subset V_h^*$ be the space of piecewise constants, and $V_{h,0}^{\perp}$ be the $L^2$ orthogonal complement of $V_{h,0}$ in $V_h^*$. The post-processed solution $u_h^*$ is defined as the solution of the system
\algn{
\label{eq:post1} (\grad_h u_h^* (t), \grad_h v) &= (\mc{A} (\sigma_h (t)+ \alpha p_h (t) \I) +  \gamma_h (t), \grad_h v) , \\
\label{eq:post2} (u_h^*, v') &= (u_h (t), v') ,
}
for all $v \in V_{h,0}^{\perp}$ and $v' \in V_{h,0}$, where $\grad_h v$ is the element-wise gradient of $v$, which is an $n\times n$ matrix-valued function. It is easy to see that $u_h^*$ is well-defined by checking that \eqref{eq:post1}--\eqref{eq:post2} is a nonsingular linear system. The error analysis of this post-processing is almost the same as the one in \cite{Arnold-Lee} but we present a detailed proof to be self-contained. 

\begin{theorem} \label{post-thm}
Suppose that the assumptions of Theorem \ref{thm:semi-err} hold and $\| u \|_{W^{1,1}H^{r}}$ is finite. Assume also that $(\sigma_{h}, u_h, \gamma_h, z_h, p_h)$ is a semidiscrete solution of \eqref{eq:disc-eq1}--\eqref{eq:disc-eq4} and $r' = r-1$. %Assume also that $u_h$ is the numerical displacement defined by \eqref{uh} with $u_h(0) \in V_h$ satisfying $\| u_h(0) - P_h u(0) \| \leq ch^{k+1} $. 
If we define $u_h^{*} (t) \in V_h^*$ by \eqref{eq:post1}--\eqref{eq:post2},
% \begin{align}
% \label{post-eq1} (\grad_h u_h^{*} (t), \grad_h \tilde{w}) &= (A_0 \sigma_{0,h} (t) + r_h(t), \grad_h \tilde{w}), & & \tilde{w} \in \tilde{V}_h, \\
% \label{post-eq2} (u_h^{*} (t), w) &= (u_h (t), w), & & w \in V_h .
% \end{align}
then $u_h^{*}(t)$ satisfies %is well-defined and there exists $c>0$ independent of $h$ such that 
\begin{align}
\label{post-error} \| u (t) - u_h^{*}(t) \|_0 \lesssim h^{r} . % \| \sigma, v, r \|_{W^{1,1}H^{k+1} \cap W^{3,1}L^2}, 
\end{align}
\end{theorem}
\begin{proof} For simplicity we will omit the time variable $t$ in the proof. 
Let $P_{h}^0$, $P_h^*$, $P_h^{\perp}$ be the $L^2$ projections into $V_{h,0}$, $V_h^*$, $V_{h,0}^{\perp}$, respectively, and $P_h^* = P_h^0 + P_h^{\perp}$ by definition. To prove \eqref{post-error}, it suffices to show 
\algn{ \label{eq:Phu-estm}
\| P_h^* u - u_h^* \|_0 \lesssim h^r
}
because $\| u - P_h^* u \|_0 \lesssim h^r$ holds by the Bramble--Hilbert lemma. To estimate $\| P_h^* u - u_h^* \|_0$, consider the orthogonal decomposition
\begin{align*}
P_h^* u - u_h^* = P_h^0 (u - u_h^* ) + P_h^{\perp} (u - u_h^* ) \in V_{h,0} \oplus V_{h,0}^{\perp} .
\end{align*}
From the facts $P_h^0 P_h^V u = P_h^0 u$ and $P_h^0 u_h^* = P_h^0 u_h$ (by the definition of $u_h^*$), we have $P_h^0 (u - u_h^* ) = P_h^0 (P_h^V u - u_h)$. Recall that 
%$\| P_h^0 (u - u_h^* ) \|_0 \leq 
$\| P_h^V u - u_h \|_0 \lesssim h^{r} $
is already obtained in Theorem~\ref{thm:semi-err}, so we only need to show $\| P_h^{\perp} (u - u_h^* ) \|_0 \lesssim h^{r}$ in order to prove \eqref{eq:Phu-estm}. Since 
\algns{
\grad_h u = \grad u = \mc{A} (\sigma + \alpha p \I) + \gamma , 
}
we have 
\algns{(\grad_h u , \grad_h w) = (\mc{A} (\sigma  + \alpha p \I) + \gamma , \grad_h w)
}
for $w \in V_{h,0}^{\perp}$. By subtracting the equation \eqref{eq:post1} from this equation, we get, for $w \in V_{h,0}^{\perp}$,
\begin{align} \label{ed-post-sub2}
(\grad_h (u - u_h^* ), \grad_h w) = (\mc{A} (e_{\sigma} + \alpha e_p \I) + e_{\gamma} , \grad_h w). %(A(\sigma_0^i - \Sigma_0^i) + r^i - R^i, \grad_h w). 
\end{align}
Using the equality
\begin{align*}
u - u_h^* &=  (u - P_h^* u ) + (P_h^* u - u_h^* )  = (u - P_h^* u ) + P_h^{\perp}( u - u_h^*) + P_h^0 (u - u_h^*)
\end{align*}
to replace $u - u_h^*$ in \eqref{ed-post-sub2}, a direct computation gives
\mltlns{
(\grad_h (P_h^{\perp} (u - u_h^*) ), \grad_h w) \\
= - (\grad_h (u - P_h^* u ) - (\mc{A} (e_{\sigma} + \alpha e_p \I) + e_{\gamma} , \grad_h w) 
}%&\quad \quad + (A_0 e_{\sigma_0}^i  + e_r^i , \grad_h w).
because $\grad_h P_h^0 ( u - u_h^*) = 0$. Taking $w = P_h^{\perp} ( u - u_h^*)$ in this equation and multiplying $h$, we have 
\begin{multline} \label{ed-post-sub3}
h \| \grad_h P_h^{\perp} (u - u_h^* ) \|_0 
\lesssim h(\| \grad_h (u - P_h^* u ) \|_0 + \| \mc{A} (e_{\sigma} + \alpha e_p \I) + e_\gamma \|_0 ).
\end{multline}
Using the estimate $\| P_h^{\perp} ( u - u_h ) \|_0 \lesssim h \| \grad_h P_h^{\perp} (u - u_h ) \|_0$, we get  
\algns{
\| P_h^{\perp} (u - u_h^*)  \|_0 \lesssim h(\| \grad_h (u - P_h^* u ) \|_0 + \| \mc{A} (e_{\sigma} + \alpha e_p \I) + e_\gamma \|_0 ).
}%&\leq ch\| \grad_h(u^i - P_h^* u^i) \| + c\| P_h u^i - U^i \| \\
%& \quad \quad + ch \| A(\sigma_0^i - \Sigma_0^i) + r^i - R^i \|, 
%where the second term on the right-hand side of last inequality is due to \eqref{equiv-ineq2}. 
Now we only need to prove that the two terms on the right-hand side of the above inequality are bounded by $ch^r$. For the first term, one can see 
\algns{
h\|\grad_h (u - P_h^* u) \|_0 \lesssim h^r \| u \|_{r} 
}
by the Bramble--Hilbert lemma. 
%For the second, we use the inverse estimate $h\| \grad_h (P_h (u - u_h ) \| \leq c \| P_h u (t) - u_h (t) \|$ and \eqref{ed-post-sub1}. 
For the second term we use the a priori error estimates of $\sigma$, $p$, $\gamma$ in Corollary~\ref{thm:full-err}, and the triangle inequality. \qed %The proof is completed. 
\end{proof}

\subsection{Well-posedness of fully discrete solutions} %\label{sec:full}
In this section we discuss the a priori error analysis of fully discrete solutions. 
%As is pointed out in the previous subsection, the system is a differential algebraic equation, so its initial data typically should satisfy some compatibility conditions. 
Let $N$ be a positive integer and $\lap t = T/N$ be the time-step size. Define $t_i = i \lap t$ for $i=0, 1, ..., N$ and 
\algn{ \label{time-quotient}
%f^{j+\half} = f \left( \frac{t_j + t_{j+1}}2 \right), \qquad 
f^{i} = f ( t_i ), \qquad  \pd_t f^{i} = \frac{f^{i} - f^{i-1}}{\lap t} . %, \qquad \hat f^{j+\half } = \frac{f^j + f^{j+1}} 2 .
}
Denoting the $i$-th time step solution by $(\sigma_h^i, u_h^i, \gamma_h^i, z_h^i, p_h^i)$, the full discretization of \eqref{eq:disc-eq1}--\eqref{eq:disc-eq4} with the backward Euler scheme is 
\begin{align}
\label{eq:fd-eq1} (\mathcal{A} (\sigma_h^{i} + \alpha p_h^{i} \I), \tau) + (u_h^{i} , \div \tau) + (\gamma_h^i, \tau) &= 0,  \\
\label{eq:fd-eq2} (\div \sigma_h^{i}, v) + (\sigma_h^{i}, \eta ) &= -(f^i, v),  \\
\label{eq:fd-eq3} (\kappa^{-1} z_h^{i}, w) + ( p_h^{i}, \div w)  &= 0, \\
\label{eq:fd-eq4} \left( s_0 \pd_t p_h^{i} , q \right) + \left( \mc{A} \left( \pd_t \sigma_h^{i} + \alpha \pd_t p_h^{i} \I \right), \alpha q \I \right) - (\div z_h^{i}, q) &= (g^i, q), 
\end{align}
for $i \geq 1$ and any $(\tau, v, \eta, w, q) \in \Sigma_h \times V_h \times \Gamma_h \times W_h \times Q_h$. This is a system of linear equations with the same number of equations and unknowns, so it is a nonsingular linear system if it has a unique solution. 
%well-posed if we show that the $i$-th time step solution vanishes when $f^i$, $g^i$, and the $(i-1)$-th time step solution vanish. 
Suppose that $f^i$, $g^i$, and the $(i-1)$-th time step solution vanish in the above system. Then we show that the $i$-th step solution must vanish. To show it, we multiply $\lap t$ to \eqref{eq:fd-eq4} and get
\algn{
\label{eq:fd-eq1a} (\mathcal{A} (\sigma_h^{i} + \alpha p_h^{i} \I), \tau) + (u_h^{i} , \div \tau) + (\gamma_h^i, \tau) &= 0, & & \forall \tau \in \Sigma_h, \\
\label{eq:fd-eq2a} (\div \sigma_h^{i}, v) + (\sigma_h^{i}, \eta ) &= 0 , & & \forall (v,\eta) \in V_h \times \Gamma_h , \\
\label{eq:fd-eq3a} (\kappa^{-1} z_h^{i}, w) + ( p_h^{i}, \div w)  &= 0, & & \forall w \in W_h, \\
\label{eq:fd-eq4a} \left( s_0 p_h^{i} , q \right) + \left( \mc{A} \left( \sigma_h^{i} + \alpha p_h^{i} \I \right), \alpha q \I \right) - \lap t (\div z_h^{i}, q) &= 0, & & \forall q \in Q_h .
}
Taking $\tau = \sigma_h^i$, $v = u_h^i$, $w = \lap t z_h^i$, $q = p_h^i$, $\eta = - \gamma_h^i$ and adding all the equations, we have 
\algns{
\| \sigma_h^{i} + \alpha p_h^{i} \I \|_{\mc{A}}^2 + (s_0 p_h^i, p_h^i) + \lap t (\kappa^{-1} z_h^i, z_h^i) = 0,
}
so $\sigma_h^{i} + \alpha p_h^{i} \I$ and $z_h^i$ vanish. Note that $p_h^i$ does not necessarily vanish because we do not assume that $s_0$ is strictly positive on the whole $\Omega$. However, the inf-sup condition {\bf (S3)}, the fact $z_h^i = 0$, and \eqref{eq:fd-eq3a} can conclude that $p_h^i = 0$, and therefore $\sigma_h^i = 0$ as well because $\sigma_h^i + \alpha p_h^i = 0$. Now $u_h^i = 0$ and $\gamma_h^i = 0$ can be obtained using the inf-sup condition {\bf (S4)} and \eqref{eq:fd-eq1a}. 

The a priori error analysis of differential algebraic equation with implicit time discretization is well-studied \cite{MR1027594}, so here we state the result and will not show the detailed proof of the error estimates of the fully discrete solutions.

\begin{theorem} \label{thm:fd-err}
Suppose that $(\sigma, u, \gamma, z, p)$ and $(\sigma_h^i, u_h^i, \gamma_h^i, z_h^i, p_h^i)$ are solutions of \eqref{eq:weak-eq1}--\eqref{eq:weak-eq4} and \eqref{eq:fd-eq1}--\eqref{eq:fd-eq4}, respectively. Then the following estimate holds:
\begin{align*}
\sup_{0 \leq i \leq N} \| \sigma^i - \sigma_h^i, u^i - u_h^i, \gamma^i - \gamma_h^i, z^i - z_h^i , p^i - p_h^i \|_0 \leq c(\lap t + h^m), 
%\| \sigma, u, z, p, \gamma \|_{W^{1,1} H^m}, 
\quad 1 \leq m \leq r,
\end{align*}
with a constant $c$ depending on regularity of the exact solution.
\end{theorem}
Let us give a remark on other time discretizations. We may use other implicit time discretization schemes with higher order time discretization errors but they may need numerical initial data compatible with algebraic equations of the problem. For example, if the Crank--Nicolson method is used and numerical initial data do not satisfy the algebraic equations of the problem, then the numerical solution does not satisfy the algebraic equations at all time steps due to the time stepping algorithm of the Crank--Nicolson method. In order to avoid it, we may choose numerical initial data as a numerical solution of the algebraic equations. An alternative way is to take only one time step with the backward Euler method with very small time-step size and to continue time stepping with the Crank--Nicolson method.

\section{Numerical results} \label{sec:numerical}
We present numerical experiments in this section. For simplicity of presentation, we assume that the poroelastic medium is homogeneous and isotropic, i.e., 
\algn{ \label{eq:stiffness}
\mc{C} \tau &= 2 \mu \tau + \lambda \tr \tau \I, \qquad \tau \in L^2(\Omega; \Bbb{S}), \\
%where $\mu$ and $\lambda$ are positive constants, $\tr \tau$ is the trace of $\tau$, and $\I$ is the identity tensor. Then $\mc{A}^s$, the inverse of $\mc{C}$, has the form
\notag \mc{A}^s \tau &= \frac{1}{2 \mu} \left( \tau - \frac{\lambda}{2 \mu + n \lambda} \tr \tau \I \right), \qquad \tau \in L^2(\Omega; \Bbb{S})
}
with positive constants $\mu$ and $\lambda$.
%the constitutive equations for $A_0$ and $A_1$ are the forms of isotropic materials introduced in \eqref{compliance} with constants $\lambda_0$, $\lambda_1$, $\mu_0$, and $\mu_1$. 

We set $\Omega = [0,1] \times [0,1]$, and displacement boundary conditions will be given in all examples. In our experiments, we used two combinations of mixed finite elements for elasticity and for mixed Poisson problems. As the first combination we use the lowest order Arnold--Falk--Winther (${\rm AFW}_1$) element \cite{AFW07} for linear elasticity and the lowest order Raviart--Thomas (${\rm RT}_1$) element for mixed Poisson problem. \tred{Note that we use the indices of Raviart--Thomas elements in the FEniCS package (cf. \cite{fenicsbook}), which may be different from other literature}. The lowest order AFW element has the lowest order Brezzi--Douglas--Marini element \cite{BDM85} in each row of $\Sigma_h$, $\V$-valued discontinuous piecewise constant polynomials as $V_h$, and $\K$-valued discontinuous piecewise constant polynomials as $\Gamma_h$. Another combination is the lowest order Taylor--Hood based (${\rm THB}_1$) element for linear elasticity and the second lowest order RT element (${\rm RT}_2$) for mixed Poisson problem. The lowest order Taylor--Hood based element \cite{Lee14a} has the lowest order BDM element in each row of $\Sigma_h$, $\V$-valued piecewise constant polynomials as $V_h$, and $\K$-valued continuous piecewise linear polynomials as $\Gamma_h$. For the stability analysis of the ${\rm AFW}_1$ and ${\rm THB}_1$ elements, we refer to \cite{AFW07,MR2449101,FalkWS,Guzman11,Lee14a}.

In our numerical experiments the mesh is structured with mesh size $h$ and we take the backward Euler time discretization with time-step size $0 < \lap t < 1$. We set $\lap t = h^2$ or $\lap t = h^3$ in order to make the convergence rate of the time discretization errors higher than the one of the spatial discretization errors, so we can compare the convergence rates of the spatial discretization errors ignoring influences of time discretization errors. The expected convergence rates of the errors from the theoretical analysis are summarized in Table \ref{err-rate}. All numerical experiments are implemented using Dolfin, the Python interface of the FEniCS package \cite{fenicsbook}.

\begin{table}[!h]  
\caption[Expected convergence rates]{Finite element spaces for unknowns (${\rm BDM}_k$ : the $k$-th lowest order Brezzi--Douglas--Marini element, ${\rm RT}_k$ : the $k$-th lowest order Raviart--Thomas element, ${\rm CG}_k$ : the Lagrange finite element with polynomials of degree $\leq k$, ${\rm DG}_k$ : the finite element with discontinuous polynomials of degree $\leq k$)} \label{fem-spaces}
\centering
\begin{tabular}{>{\small}c >{\small}c >{\small}c >{\small}c >{\small}c >{\small}c >{\small}c >{\small}c >{\small}c >{\small}c >{\small}c }
\hline
& {$\Sigma_h$} & {$V_h$} & {$\Gamma_h$} & {$W_h$} & {$Q_h$} \\ \hline
%{norm} & ${L^\infty L^2}$ & $L^\infty L^2$ & ${L^\infty L^2}$ &  ${L^\infty L^2}$ &  ${L^\infty L^2}$  &  ${L^\infty L^2}$ \\
\hline
{Element 1} & ${\rm BDM}_1$ & ${\rm DG}_0$ & ${\rm DG}_0$ & ${\rm RT}_1$ &  ${\rm DG}_0$ \\ %&  --  & 4.18e-3 &  --  & 2.65e-2 &  -- \\
{Element 2}& ${\rm BDM}_1$ & ${\rm DG}_0$ & ${\rm CG}_1$ & ${\rm RT}_2$ &  ${\rm DG}_1$ \\ \hline %& 2.00 & 8.14e-6 & 2.00 & 1.04e-4 & 2.00\\ \hline
\end{tabular}         
\end{table}

\begin{table}[!h]  
\caption[Expected convergence rates]{Expected (spatial) convergence rates} \label{err-rate}
\centering
\begin{tabular}{>{\small}c >{\small}c >{\small}c >{\small}c >{\small}c >{\small}c >{\small}c >{\small}c >{\small}c >{\small}c >{\small}c >{\small}c}
\hline
{error} & {$\sigma - \sigma_h$} & {$ u - u_h $} & {$ u - u_h^* $} & {$ \gamma - \gamma_h$ } & {$ z - z_h $} & {$ p - p_h $} \\ \hline
{norm} & ${L^\infty L^2}$ & $L^\infty L^2$ & ${L^\infty L^2}$ &  ${L^\infty L^2}$ &  ${L^\infty L^2}$  &  ${L^\infty L^2}$ \\
\hline
{Element 1} & 1 & 1 & 1 & 1 & 1 & 1 \\ %&  --  & 4.18e-3 &  --  & 2.65e-2 &  -- \\
{Element 2} & 2 & 1 & 2 & 2 & 2 & 2 \\ \hline %& 2.00 & 8.14e-6 & 2.00 & 1.04e-4 & 2.00\\ \hline
\end{tabular}         
\end{table}

\begin{table}[!h]  
\setlength{\tabcolsep}{4.8pt} 
\caption[Numerical result with Element 1]{Errors and convergence rates of unknowns at $t = 1.0$ with Element 1 for the exact solution with the displacement \eqref{ex1} ($\mu = \lambda = 10$, $s_0 = 1$, $\lap t = h^2$).} \label{AFW-result1}
\begin{tabular}{>{\tiny}c >{\tiny }c >{\tiny}c >{\tiny}c >{\tiny}c >{\tiny}c >{\tiny}c >{\tiny}c >{\tiny}c >{\tiny}c >{\tiny}c >{\tiny}c >{\tiny}c}
\hline
\multirow{2}{*}{$\frac 1h$} & \multicolumn{2}{>{\tiny}c}{$\| \sigma - \sigma_h \|$} & \multicolumn{2}{>{\tiny}c}{$\| u - u_h \|$} & \multicolumn{2}{>{\tiny}c}{$\| u - u_h^* \|$} & \multicolumn{2}{>{\tiny}c}{$\| \gamma - \gamma_h \|$} & \multicolumn{2}{>{\tiny}c}{$\| z - z_h \|$} & \multicolumn{2}{>{\tiny}c}{$\| p - p_h \|$} \\ 
 & error & rate & error & rate & error & rate& error & rate& error & rate& error & rate\\ \hline
4 & 5.93e+0 &  --  & 1.92e--1 &  --  & 1.10e--1 &  --  & 2.50e--1 &  --  & 1.51e--1 &  --  & 6.24e--2 &  -- \\
8 & 2.37e+0 & 1.32 & 8.94e--2 & 1.10 & 3.54e--2 & 1.64 & 1.16e--1 & 1.11 & 7.78e--2 & 0.95 & 3.19e--2 & 0.97\\
16& 1.10e+0 & 1.11 & 4.26e--2 & 1.07 & 1.03e--2 & 1.78 & 5.53e--2 & 1.06 & 3.90e--2 & 1.00 & 1.55e--2 & 1.05\\
32& 5.40e--1 & 1.03 & 2.09e--2 & 1.03 & 2.77e--3 & 1.90 & 2.72e--2 & 1.02 & 1.95e--2 & 1.00 & 7.59e--3 & 1.03\\
64& 2.69e--1 & 1.00 & 1.04e--2 & 1.01 & 7.15e--4 & 1.95 & 1.35e--2 & 1.01 & 9.75e--3 & 1.00 & 3.77e--3 & 1.01\\ \hline
\end{tabular}
\end{table}

\begin{table}[!h]  
\setlength{\tabcolsep}{4.8pt} 
\caption[Numerical result with Element 2]{Errors and convergence rates of unknowns at $t = 1.0$ with Element 2 for the exact solution with the displacement \eqref{ex1} ($\mu = \lambda = 10$, $s_0 = 1$, $\lap t = h^3$).} \label{THB-result1}
\begin{tabular}{>{\tiny}c >{\tiny }c >{\tiny}c >{\tiny}c >{\tiny}c >{\tiny}c >{\tiny}c >{\tiny}c >{\tiny}c >{\tiny}c >{\tiny}c >{\tiny}c >{\tiny}c}
\hline
\multirow{2}{*}{$\frac 1h$} & \multicolumn{2}{>{\tiny}c}{$\| \sigma - \sigma_h \|$} & \multicolumn{2}{>{\tiny}c}{$\| u - u_h \|$} & \multicolumn{2}{>{\tiny}c}{$\| u - u_h^* \|$} & \multicolumn{2}{>{\tiny}c}{$\| \gamma - \gamma_h \|$} & \multicolumn{2}{>{\tiny}c}{$\| z - z_h \|$} & \multicolumn{2}{>{\tiny}c}{$\| p - p_h \|$} \\ 
 & error & rate & error & rate & error & rate& error & rate& error & rate& error & rate\\ \hline
4 & 3.19e+0 &  --  & 1.62e--1 &  --  & 3.63e--2 &  --  & 2.23e--1 &  --  & 2.61e--2 &  --  & 1.09e--2 &  -- \\
8 & 7.97e--1 & 2.00 & 8.26e--2 & 0.98 & 8.13e--3 & 2.16 & 4.20e--2 & 2.41 & 7.08e--3 & 1.88 & 2.84e--3 & 1.93\\
16& 2.03e--1 & 1.97 & 4.14e--2 & 1.00 & 2.17e--3 & 1.91 & 9.63e--3 & 2.13 & 1.87e--3 & 1.92 & 7.61e--4 & 1.90\\
32& 5.14e--2 & 1.98 & 2.07e--2 & 1.00 & 5.55e--4 & 1.96 & 2.34e--3 & 2.04 & 4.77e--4 & 1.97 & 2.00e--4 & 1.93\\
64& 1.29e--2 & 1.99 & 1.04e--2 & 1.00 & 1.40e--4 & 1.99 & 5.79e--4 & 2.01 & 1.20e--4 & 1.99 & 5.13e--5 & 1.96\\ \hline
\end{tabular}
\end{table}

\begin{table}[!h]  
\setlength{\tabcolsep}{4.8pt} 
\caption[Numerical result with Element 1]{Errors and convergence rates of unknowns at $t = 1.0$ with Element 1 for the exact solution with the displacement \eqref{ex1} ($\mu = \lambda = 10$, $s_0 = 10^{-3}$, $\lap t = h^2$).} \label{AFW-result2}
\begin{tabular}{>{\tiny}c >{\tiny }c >{\tiny}c >{\tiny}c >{\tiny}c >{\tiny}c >{\tiny}c >{\tiny}c >{\tiny}c >{\tiny}c >{\tiny}c >{\tiny}c >{\tiny}c}
\hline
\multirow{2}{*}{$\frac 1h$} & \multicolumn{2}{>{\tiny}c}{$\| \sigma - \sigma_h \|$} & \multicolumn{2}{>{\tiny}c}{$\| u - u_h \|$} & \multicolumn{2}{>{\tiny}c}{$\| u - u_h^* \|$} & \multicolumn{2}{>{\tiny}c}{$\| \gamma - \gamma_h \|$} & \multicolumn{2}{>{\tiny}c}{$\| z - z_h \|$} & \multicolumn{2}{>{\tiny}c}{$\| p - p_h \|$} \\ 
 & error & rate & error & rate & error & rate& error & rate& error & rate& error & rate\\ \hline
4 & 5.93e+00 &  --  & 1.92e--1 &  --  & 1.10e--1 &  --  & 2.50e--1 &  --  & 1.67e--1 &  --  & 7.82e--2 &  -- \\
8 & 2.37e+00 & 1.32 & 8.94e--2 & 1.10 & 3.54e--2 & 1.64 & 1.16e--1 & 1.11 & 8.12e--2 & 1.05 & 3.55e--2 & 1.14\\
16& 1.10e+00 & 1.11 & 4.26e--2 & 1.07 & 1.03e--2 & 1.78 & 5.53e--2 & 1.06 & 3.95e--2 & 1.04 & 1.60e--2 & 1.15\\
32& 5.40e-01 & 1.03 & 2.09e--2 & 1.03 & 2.77e--3 & 1.90 & 2.72e--2 & 1.02 & 1.96e--2 & 1.01 & 7.66e--3 & 1.06\\
64& 2.69e-01 & 1.00 & 1.04e--2 & 1.01 & 7.14e--4 & 1.95 & 1.35e--2 & 1.01 & 9.76e--3 & 1.00 & 3.78e--3 & 1.02\\ \hline
\end{tabular}
\end{table}

\begin{table}[!h]  
\setlength{\tabcolsep}{4.8pt} 
\caption[Numerical result with Element 2]{Errors and convergence rates of unknowns at $t = 1.0$ with Element 2 for the exact solution with the displacement \eqref{ex1} ($\mu = \lambda = 10$, $s_0 = 10^{-3}$, $\lap t = h^3$).} \label{THB-result2}
\begin{tabular}{>{\tiny}c >{\tiny }c >{\tiny}c >{\tiny}c >{\tiny}c >{\tiny}c >{\tiny}c >{\tiny}c >{\tiny}c >{\tiny}c >{\tiny}c >{\tiny}c >{\tiny}c}
\hline
\multirow{2}{*}{$\frac 1h$} & \multicolumn{2}{>{\tiny}c}{$\| \sigma - \sigma_h \|$} & \multicolumn{2}{>{\tiny}c}{$\| u - u_h \|$} & \multicolumn{2}{>{\tiny}c}{$\| u - u_h^* \|$} & \multicolumn{2}{>{\tiny}c}{$\| \gamma - \gamma_h \|$} & \multicolumn{2}{>{\tiny}c}{$\| z - z_h \|$} & \multicolumn{2}{>{\tiny}c}{$\| p - p_h \|$} \\ 
 & error & rate & error & rate & error & rate& error & rate& error & rate& error & rate\\ \hline
4 & 3.19e+ 0 &  --  & 1.62e--1 &  --  & 3.63e--2 &  --  & 2.23e--1 &  --  & 2.65e--2 &  --  & 1.24e--2 &  -- \\
8 & 7.97e--1 & 2.00 & 8.26e--2 & 0.98 & 8.13e--3 & 2.16 & 4.20e--2 & 2.41 & 6.68e--3 & 1.99 & 2.97e--3 & 2.06\\
16& 2.03e--1 & 1.97 & 4.14e--2 & 1.00 & 2.17e--3 & 1.91 & 9.63e--3 & 2.13 & 1.75e--3 & 1.93 & 7.37e--4 & 2.01\\
32& 5.14e--2 & 1.98 & 2.07e--2 & 1.00 & 5.55e--4 & 1.96 & 2.34e--3 & 2.04 & 4.46e--4 & 1.98 & 1.85e--4 & 2.00\\
64& 1.29e--2 & 1.99 & 1.04e--2 & 1.00 & 1.40e--4 & 1.99 & 5.79e--4 & 2.01 & 1.12e--4 & 1.99 & 4.63e--5 & 2.00\\ \hline
\end{tabular}
\end{table}

\begin{table}[!h] 
\setlength{\tabcolsep}{4.8pt} 
\caption[Numerical result with Element 1]{Errors and convergence rates of unknowns at $t = 1.0$ with Element 1 for the exact solution with the displacement \eqref{ex1} ($\mu = \lambda = 10$, $s_0 = 0$, $\lap t = h^3$).} \label{AFW-result3}
\begin{tabular}{>{\tiny}c >{\tiny }c >{\tiny}c >{\tiny}c >{\tiny}c >{\tiny}c >{\tiny}c >{\tiny}c >{\tiny}c >{\tiny}c >{\tiny}c >{\tiny}c >{\tiny}c}
\hline
\multirow{2}{*}{$\frac 1h$} & \multicolumn{2}{>{\tiny}c}{$\| \sigma - \sigma_h \|$} & \multicolumn{2}{>{\tiny}c}{$\| u - u_h \|$} & \multicolumn{2}{>{\tiny}c}{$\| u - u_h^* \|$} & \multicolumn{2}{>{\tiny}c}{$\| \gamma - \gamma_h \|$} & \multicolumn{2}{>{\tiny}c}{$\| z - z_h \|$} & \multicolumn{2}{>{\tiny}c}{$\| p - p_h \|$} \\ 
 & error & rate & error & rate & error & rate& error & rate& error & rate& error & rate\\ \hline
4 & 5.93e+0  &  --  & 1.92e--1 &  --  & 1.10e--1 &  --  & 2.50e--1 &  --  & 1.67e--1 &  --  & 7.82e--2 &  -- \\
8 & 2.37e+0  & 1.32 & 8.93e--2 & 1.10 & 3.54e--2 & 1.64 & 1.15e--1 & 1.11 & 8.12e--2 & 1.05 & 3.55e--2 & 1.14\\
16& 1.10e+0  & 1.11 & 4.26e--2 & 1.07 & 1.03e--2 & 1.78 & 5.53e--2 & 1.06 & 3.95e--2 & 1.04 & 1.60e--2 & 1.15\\
32& 5.40e--1 & 1.03 & 2.09e--2 & 1.03 & 2.77e--3 & 1.90 & 2.72e--2 & 1.02 & 1.96e--2 & 1.01 & 7.66e--3 & 1.06\\
64& 2.69e--1 & 1.00 & 1.04e--2 & 1.01 & 7.14e--4 & 1.95 & 1.35e--2 & 1.01 & 9.76e--3 & 1.00 & 3.78e--3 & 1.02\\ \hline
\end{tabular}
\end{table}

\begin{table}[!h]
\setlength{\tabcolsep}{4.8pt} 
\caption[Numerical result with Element 2]{Errors and convergence rates of unknowns at $t = 1.0$ with Element 2 for the exact solution with the displacement \eqref{ex1} ($\mu = \lambda = 10$, $s_0 = 0$, $\lap t = h^3$).} \label{THB-result3}
\begin{tabular}{>{\tiny}c >{\tiny }c >{\tiny}c >{\tiny}c >{\tiny}c >{\tiny}c >{\tiny}c >{\tiny}c >{\tiny}c >{\tiny}c >{\tiny}c >{\tiny}c >{\tiny}c}
\hline
\multirow{2}{*}{$\frac 1h$} & \multicolumn{2}{>{\tiny}c}{$\| \sigma - \sigma_h \|$} & \multicolumn{2}{>{\tiny}c}{$\| u - u_h \|$} & \multicolumn{2}{>{\tiny}c}{$\| u - u_h^* \|$} & \multicolumn{2}{>{\tiny}c}{$\| \gamma - \gamma_h \|$} & \multicolumn{2}{>{\tiny}c}{$\| z - z_h \|$} & \multicolumn{2}{>{\tiny}c}{$\| p - p_h \|$} \\ 
 & error & rate & error & rate & error & rate& error & rate& error & rate& error & rate\\ \hline
4 & 3.19e+0  &  --  & 1.62e--1 &  --  & 3.63e--2 &  --  & 2.23e--1 &  --  & 2.65e--2 &  --  & 1.24e--2 &  -- \\
8 & 7.97e--1 & 2.00 & 8.26e--2 & 0.98 & 8.13e--3 & 2.16 & 4.20e--2 & 2.41 & 6.68e--3 & 1.99 & 2.97e--3 & 2.07\\
16& 2.03e--1 & 1.97 & 4.14e--2 & 1.00 & 2.17e--3 & 1.91 & 9.63e--3 & 2.13 & 1.75e--3 & 1.93 & 7.37e--4 & 2.01\\
32& 5.14e--2 & 1.98 & 2.07e--2 & 1.00 & 5.55e--4 & 1.96 & 2.34e--3 & 2.04 & 4.46e--4 & 1.98 & 1.85e--4 & 2.00\\
64& 1.29e--2 & 1.99 & 1.04e--2 & 1.00 & 1.40e--4 & 1.99 & 5.79e--4 & 2.01 & 1.12e--4 & 1.99 & 4.63e--5 & 2.00\\ \hline
\end{tabular}
\end{table}

\begin{eg} \normalfont
For the displacement and pressure 
\begin{align} \label{ex1}
u(t,x,y) =
\pmat{ x \cos (t) \\
 (1+y^2) \cos(t+1) \sin(\pi x) 
} \quad \text{and}\quad p(t,x,y) = x^2 y \cos(t^2), 
\end{align}
the stiffness tensor \eqref{eq:stiffness} and given parameters $\mu$, $\lambda$, $\kappa$, $\alpha$, one can compute $\sigma$, $z$, $\gamma$, $f$, and $g$, using the equations \eqref{eq:4-field-eq1}--\eqref{eq:4-field-eq4} and the definition of $\gamma$. We compute a numerical solution of this exact solution with inhomogeneous displacement boundary conditions using the formulation \eqref{eq:weak-eq1}--\eqref{eq:weak-eq4}. The numerical results with Element~1 $({\rm AFW}_1$ and ${\rm RT}_1)$ and Element 2 $({\rm THB}_1$ and ${\rm RT}_2)$ are given through {\rm Table\tblu{s}~\ref{AFW-result1}}, 4, 5 and 6. The parameter values, mesh and time-step sizes are explained in the tables.

\begin{table}[ht]  \label{table:rel-locking}
\caption{Relative $L^2$ errors and convergence rates of $\sigma$, $u$, and $z$ for large $\lambda$ values}
%in \eqref{eq:3f-ps-precond}. ($\Omega$ = unit square, partitioned as bisections of $N \times N$ rectangles, convergence criterion with relative residual of $10^{-6}$)}} \label{3f-ps-exp}
\centering
\begin{tabular}{>{\small}c >{\small}c | >{\small}c >{\small}c >{\small}c >{\small}c >{\small}c >{\small}c }
\hline 
%\multirow{2}{*}{elements}	
%& & \multicolumn{6}{>{\small}c}{$k$ \quad$(\kappa = 10^{k})$} \\ %\hline %& \multirow{2}{*}{order } & \multirow{2}{*}{$n$ } \\ 
\multirow{2}{*}{$\lambda$} & \multirow{2}{*}{$\frac 1h$} & \multicolumn{2}{>{\small }c}{$\| \sigma - \sigma_h \|/ \| \sigma \|$} & \multicolumn{2}{>{\small }c}{$\| u - u_h \|/ \| u \| $} & \multicolumn{2}{>{\small }c}{$\| z - z_h \|/ \| z \|$} \\ %& \multicolumn{2}{>{\tiny}c}{$\| \gamma - \gamma_h \|$} & \multicolumn{2}{>{\tiny}c}{$\| z - z_h \|$} & \multicolumn{2}{>{\tiny}c}{$\| p - p_h \|$} \\ 
& & error & rate& error & rate& error & rate\\ \hline
%$\lambda$& $\frac{1}{h}$& $0$ & $-1$ & ${-2}$ & $-3$ & ${-4}$ & $-5$  \\  
\hline
\multirow{4}{*}{$10^1$} 
& $ 4$ & 4.51e--01  &  --  & 3.46e-01  &  --  & 4.09e-01 &  --  \\ 
& $ 8$ & 2.94e--01  & 0.62 & 1.53e-01  &  1.17& 2.24e-01 & 0.87 \\ 
& $16$ & 1.80e--01  & 0.71 & 5.94e-02  &  1.37& 1.14e-01 & 0.98 \\  
& $32$ & 1.04e--01  & 0.78 & 2.11e-02  &  1.49& 5.52e-02 & 1.04 \\ \hline \multirow{4}{*}{$10^4$}
& $ 4$ & 6.23e--01  &  --  & 5.88e-01  &   -- & 4.82e-01 &  --  \\ 
& $ 8$ & 4.62e--01  & 0.43 & 3.19e-01  &  0.88& 2.54e-01 & 0.92 \\ 
& $16$ & 3.15e--01  & 0.55 & 1.50e-01  &  1.09& 1.26e-01 & 1.01 \\  
& $32$ & 1.98e--01  & 0.67 & 6.12e-02  &  1.29& 5.98e-02 & 1.07 \\ \hline \multirow{4}{*}{$10^7$}
& $ 4$ & 6.24e--01  &  --  & 5.89e-01  &   -- & 4.82e-01 &  --  \\ 
& $ 8$ & 4.63e--01  & 0.43 & 3.20e-01  &  0.88& 2.55e-01 & 0.92 \\ 
& $16$ & 3.16e--01  & 0.55 & 1.50e-01  &  1.09& 1.26e-01 & 1.01 \\  
& $32$ & 1.98e--01  & 0.67 & 6.14e-02  &  1.29& 5.98e-02 & 1.07 \\ \hline \multirow{4}{*}{$10^{10}$}
& $ 4$ & 6.24e--01  &  --  & 5.89e-01  &   -- & 4.82e-01 &  --  \\ 
& $ 8$ & 4.63e--01  & 0.43 & 3.20e-01  &  0.88& 2.55e-01 & 0.92 \\ 
& $16$ & 3.16e--01  & 0.55 & 1.50e-01  &  1.09& 1.26e-01 & 1.01 \\  
& $32$ & 1.98e--01  & 0.67 & 6.14e-02  &  1.29& 5.98e-02 & 1.07 \\ \hline 
\end{tabular}                           
\end{table}

In {\rm Table~\ref{AFW-result1}} and {\rm Table~\ref{AFW-result2}}, with different values of $s_0$, we carried out the local post-processing in the previous section. The post-processed solutions show second order convergence but this superconvergence is not covered in the error analysis. Numerical experiments of the same exact solutions with Element~2 are presented in {\rm Table~\ref{THB-result1}} and {\rm Table~\ref{THB-result2}} and all the convergence rates are in agreement with the expected convergence rates. In {\rm Table~\ref{AFW-result3}} and {\rm Table~\ref{THB-result3}}, we carried out numerical experiments for an exact solution with $s_0 = 0$, and can see that convergence rates are not influenced by this vanishing $s_0$.
\end{eg}

\begin{eg} \normalfont
In order to illustrate that our methods are robust for nearly incompressible materials, we consider a problem on $\Omega = [0,1] \times [0,1]$ with 
\algns{ f = 
\pmat{ x y \\ \sin t
}, \quad \mu = 10, \quad \kappa = 1, \quad s_0 = 10^{-3} ,
}
and boundary conditions
\algns{ 
\sigma \bs{n} = 0, \quad z \cdot \bs{n} = 0, \qquad \text{ on } \Gamma = \{ (x,y) \in \R^2 \,:\, y < 1 \}. 
}
Since we do not know the exact solution, we compute a numerical solution with \tblu{the mesh} of $\Omega$ bisecting $128 \times 128$ rectangles, and use it to compute the errors of other numerical solutions with coarser meshes. For simplicity, we use Element~1 in Table~\ref{err-rate}. We present relative $L^2$ errors and convergence rates of $\sigma$, $u$, $z$ for different $\lambda$ values and mesh refinements in Table~\ref{table:rel-locking}. Due to the limit of computational resources, our numerical experiments did not reach the asymptotic regime of convergence rates but they clearly show that relative $L^2$ errors of $\sigma$, $u$, $z$, are not influenced by large $\lambda$ values. 
\end{eg}

\section{Conclusion} \label{sec:conclusion}
In the paper, we propose a new finite element method for Biot's consolidation model and show the a priori error estimates of semidiscrete problems. In particular, our error estimates do not require strictly positive $s_0$, and they are robust for nearly incompressible materials.  %For fully discrete solutions with the Crank--Nicolson time discretization we obtained optimal uniform-in-time error estimates of all the unknowns. Our error analysis does not require uniformly positive storage coefficient $c_0$. 
% %because  without the assumption that the specific storage coefficient is uniformly positive. 
% Moreover, we do not use Gr\"{o}nwall's inequality in our error analysis, so there is no exponentially growing factor in the obtained error bounds. 
We illustrate the validity of our analysis by numerical experiments. 
%The computer implementation of this method is under development and numerical results will be reported in our forthcoming works.

\providecommand{\bysame}{\leavevmode\hbox to3em{\hrulefill}\thinspace}
\providecommand{\MR}{\relax\ifhmode\unskip\space\fi MR }
% \MRhref is called by the amsart/book/proc definition of \MR.
\providecommand{\MRhref}[2]{%
  \href{http://www.ams.org/mathscinet-getitem?mr=#1}{#2}
}
\providecommand{\href}[2]{#2}

\providecommand{\bysame}{\leavevmode\hbox to3em{\hrulefill}\thinspace}
\providecommand{\MR}{\relax\ifhmode\unskip\space\fi MR }
% \MRhref is called by the amsart/book/proc definition of \MR.
\providecommand{\MRhref}[2]{%
  \href{http://www.ams.org/mathscinet-getitem?mr=#1}{#2}
}

%\bibliography{./../reference/FEM}
%\bibliographystyle{amsplain}
\vspace{.125in}

\end{document}